\documentclass[11pt]{article}
\usepackage{amsfonts}
\usepackage{amsmath}
\usepackage{amsthm,nicefrac} 
\usepackage{graphicx}
\usepackage{mathrsfs,booktabs}
\usepackage{chicago,overpic}
\usepackage[small]{caption}

\newtheorem{proposition}{Proposition}[section]
\newtheorem{lemma}[proposition]{Lemma}
\newtheorem{theorem}[proposition]{Theorem}

\theoremstyle{definition}
\newtheorem{definition}[proposition]{Definition}
\newtheorem{example}[proposition]{Example}
\newtheorem{algorithm}[proposition]{Algorithm}

\theoremstyle{remark}

\newcommand{\e}{{\rm e}}
\newcommand{\E}{{\mathbb E}}
\newcommand{\bbF}{{\mathbb F}}

\renewcommand{\P}{{\mathbb P}}
\newcommand{\Q}{{\mathbb Q}}

\newcommand{\R}{{\mathbb R}}
\newcommand{\N}{{\mathbb N}}

\renewenvironment{description}{\list{}{
\leftmargin 12pt \itemindent8pt
}
}{
  \endlist
}

\newcommand{\Ind}{{ 1}}
\newcommand{\ind}[1]{\Ind_{\{#1\}}}

\newcommand{\cA}{{\mathcal A}}

\newcommand{\cB}{{\mathcal B}}

\newcommand{\cC}{{\mathcal C}}
\newcommand{\cF}{{\mathcal F}}

\newcommand{\cN}{{\mathcal N}}

\renewcommand{\ind}[1]{\Ind_{\{#1\}}}

\newcommand{\ccL}{\mathscr{L}}

\newenvironment{enumeratei}{\begin{enumerate} 
}{
  \end{enumerate}
}

\newcommand{\U}{\Upsilon}

\renewcommand{\e}{e}

\newcommand{\Nspace}{\mathcal{N}^2(0,T;{H})}
\newcommand{\Gene}{\mathscr{L}}

\begin{document}
\begin{center}
\large {\Large \bf On Galerkin Approximations for the Zakai Equation\\
with Diffusive and Point Process Observations}\\[0.3cm] %
{\sc R\"udiger Frey\footnote{Institute for Statistics and mathematics, WU Vienna, mail:
ruediger.frey@wu.ac.at},
Thorsten Schmidt\footnote{Department of mathematics, Chemnitz University of Technology, mail:
thorsten.schmidt@mathematik.tu-chemnitz.de} and Ling Xu\footnote{School of Economics and Administration, China University of Petroleum, mail: Ling.Xu@math.uni-leipzig.de. Part of this work was written while the first and the third author were at the department of mathematics, University of Leipzig.
Financial support from the International Max Planck Research School `Mathematics in the
Sciences', Leipzig and from the German Science foundation (DFG) is gratefully acknowledged. We thank two anonymous referees for their valuable and helpful comments.}}\\[0.4cm]
{\it \today}\\[0.9cm]
\end{center}

\abstract{This paper studies Galerkin approximations  applied to the  Zakai
equation of stochastic  filtering. The basic idea of this approach is to project the
infinite-dimensional Zakai equation  onto some finite-dimensional subspace generated by smooth basis functions; this leads to a finite-dimensional system of stochastic differential equations that can be solved
numerically. The contribution of the paper is twofold. On the theoretical side, existing
convergence results are extended to filtering models with observations of point-process
or mixed  type. On the applied side, various issues related to the numerical
implementation of the method are considered; in particular, we propose to work with  a subspace that is constructed from a basis of Hermite polynomials.   The paper closes with a numerical case study. }



\section{Introduction}

Stochastic  filtering deals with the recursive estimation of the
conditional distribution of a signal process $X$ given some form of noisy
observation of $X$. In the standard continuous time filtering models this
noisy observation is generated by a process $Z$ with dynamics of the form
\begin{equation} \label{eq:diffusion-observation}
Z_t = Z_0 + \int_0^t h(X_s) ds +W_t
\end{equation}
for some Brownian motion $W$ that is independent of $X$. In that case $\pi_t(dx)$, the
conditional distribution of $X_t$ given $\sigma(Z_s\colon s \le t)$, can be characterized
by a measure-valued stochastic partial differential equation (SPDE) known  as
\emph{Zakai equation}. This  SPDE is in general an  infinite-dimensional equation that
cannot be solved directly.  In view of the practical relevance of filtering, a wide range
of methods for the approximation of this equation by finite-dimensional systems and for
the numerical solution of filtering problems in general has therefore  been developed; a
good survey is given in \citeN{bib:budhiraja-chen-lee-07} or in
\citeN{bib:bain-crisan-08}. Popular numerical methods for filtering problems include the
extended Kalman filter (\citeN{bib:jazwinski-70}); quantization
(\citeN{bib:gobet-et-al-06}); Markov-chain approximation (\citeN{bib:kushner-dupuis-01},
\citeN{bib:dimasi-runggaldier-82}); spectral methods (\citeN{bib:lototsky-06}) and
simulation methods such as particle filtering  (\citeN{bib:crisan-moral-lyons-99}).

If the signal $X$ is a diffusion process with uniformly parabolic generator the
conditional distribution $\pi_t(dx)$  admits a Lebesgue density that solves a SPDE in a
suitable function space, the so-called Zakai equation for the unnormalized conditional
density; see for instance \citeN{bib:pardoux-79}. Galerkin approximations are an
important numerical technique for dealing with this SPDE. The basic idea of this approach
is to project the Zakai equation for the conditional density onto some finite-dimensional
subspace $H_n$ generated by  basis functions $e_1, \dots, e_n$. This leads to an
$n$-dimensional SDE system for the Fourier coefficients of the solution of the  projected
equation; this SDE system can then be solved by numerical methods for ``ordinary'' SDEs.

Theoretical and numerical aspects of Galerkin approximations are well understood for the
case of pure diffusion observation as in \eqref{eq:diffusion-observation}; see for
instance  \citeN{bib:germani-piccioni-84} and \citeN{bib:germani-piccioni-87} for
convergence results for Galerkin approximations and \citeN{bib:ahmed-radaideh-97} or \citeN{bib:ahmed-99} for a
case study and a discussion of numerical aspects.
Much less is known for the case of
mixed observations of diffusion \emph{and} point-process type. In this paper we therefore
consider a model where a doubly stochastic point process $N$ with intensity
$\lambda(X_t)$ is observable in addition to the process $Z$. Models of this type arise
naturally in credit risk modelling (see Example~\ref{ex:credit-risk} below) or in the
modelling of high frequency data in finance (\citeN{bib:frey-runggaldier-01},
\citeN{bib:cvitanic-liptser-rozovski-06}). Outside the field of financial mathematics
point-process information plays among others a crucial role in the analysis of queueing
systems (\citeN{bib:bremaud-81}).

Our contribution is twofold. On the theoretical side we generalize the convergence
results of \citeN{bib:germani-piccioni-87} to the case of mixed observations.\footnote{In this context we mention the recent work \citeN{bib:hausenblas-08} where theoretical properties of a finite element approximation of certain SPDEs driven by a Poisson random measure of pure jump type are studied.}
On the applied side we
extend the numerical analysis of \citeN{bib:ahmed-radaideh-97} in various ways: to begin
with, we propose to use Hermite polynomials as basis functions (instead of Gaussian basis
functions);  we explain how to change the basis adaptively in order to deal with sudden
shifts in location and scale of the conditional density caused for instance by jumps in
the observation, and we compare several methods for the numerical implementation of the
SDE-system that results from the Galerkin approximation. An extensive simulation study
shows that the Galerkin approximation works well for systems with mixed observation
provided that the necessary care is taken in the implementation of the method.

The paper is organized as follows. The model and the various versions of the Zakai
equation are described in Section~\ref{sec:model-and-Zakai-eqn}. In that section we
moreover introduce the basic form of the Galerkin approximation. Convergence results for
the Galerkin approximation are given in Section~\ref{sec:convergence}.
Section~\ref{sec:numerical-methods} deals with the numerical implementation of the model;
results from numerical experiments are finally reported in Section~\ref{Simulation
Results}.

\section{Zakai equation and  Galerkin approximation}
\label{sec:model-and-Zakai-eqn}

In this section we introduce the  nonlinear filtering problem  studied  in this
paper. Moreover, we present  different versions of the Zakai equation that describe the
solution of the filtering problem.  Finally we introduce  the Galerkin approximation for
the Zakai equation for the unnormalized conditional density and we derive an SDE system
for the Fourier coefficients.

\subsection{Model and notation}

We consider a filtered probability space $(\Omega, \mathcal{F},\mathbb{F}, \mathbb{P})$
where the filtration $\mathbb{F}=(\mathcal{F}_t)_{0 \le t \le T}$ satisfies the usual
conditions and where $T$ is an arbitrary but fixed horizon date. The nonlinear filtering
problem we study consists of an unobserved state process $X$ and observations $Z$ and
$N$. $Z$ is a nonlinear continuous transformation of $X$  with additional Gaussian noise,
while $N$ is a doubly stochastic Poisson process whose intensity is a nonlinear
function of $X$. 

\paragraph{The state process.}

We consider  an unobserved \emph{state process} $X$ on $\R^d$ which is the solution of
the SDE
\begin{align}\label{eq:state}
X_t=X_0+\int_0^t b(X_s)ds +\int_0^t \sigma(X_s)dV_s, \quad 0 \leq t \leq
T,
\end{align}
for a  $m$-dimensional $\mathbb{F}$-Brownian motion $V$. Moreover, we assume that $X_0$
has finite second moments and a density $p_0\in L^2(\R^d)$.   Set
$a(x)=\sigma(x)\sigma(x)^\top$. The components of $a(x)$ and $b(x)$ are denoted by
$a_{ij}(x)$ and $b_i(x)$, respectively.  The restriction of the generator $\Gene$ of the
Markov process $X$ to $C^2_b (\R^d)$, the set of all bounded and twice continuously
differentiable functions on $\R^d$, is given by the second order differential operator
\begin{align}\label{eq:A}
\Gene=\sum_{i=1}^d b_i(x) \frac{\partial }{\partial x_i}+\frac{1}{2}
\sum_{i,j=1}^d a_{ij}(x) \frac{\partial^2}{\partial x_i \partial x_j}.
\end{align}
Note that the It\^o-formula implies that for  $ f \in C^2_b(\R^d)$,
$M^f_t:=f(X_t)-f(X_0)-\int_0^t \Gene f(X_s)ds$ is an $\bbF$-martingale.

\paragraph{The observation processes.}
The observation is given by the two processes $Z$ and $N$. The process   $Z$
satisfies
\begin{align}\label{eq:Z} Z_t=\int_0^t h(X_s)ds +W_t, \quad 0\leq t
< \infty \,,
\end{align}
where  $h:\R^d\rightarrow \R^l$ is a measurable function and  $W$ is an $l$-dimensional
standard Brownian motion, independent of $X$. Moreover, the process $N$ is a doubly
stochastic Poisson process with intensity $\lambda(X_t)$ where $\lambda$ is a positive,
continuous and  bounded function, so that  the process $ N_t - \int_0^t \lambda(X_s) ds $
is an $\mathbb{F}$-martingale. We denote  the  jump times  of $N$ by $\tau_1,
\tau_2,\dots$.

The objective of nonlinear filtering is to find suitable ways for computing
$\pi_t(dx)$,  the conditional distribution of the state $X_t$ given the
observation history in a recursive way.   More formally, let $\cF^{Z,N}_{t} :=
\sigma( Z_u, N_u :0\leq u\leq t)$, so that the associated filtration
$\bbF^{Z,N}$ represents the information given by the observation. The
conditional distribution of $X_t$ given the observation until time $t$ is
determined by
\begin{align*}
\pi_t(f) := \E\big(f(X_t)|\cF^{Z,N}_{t} \big),\quad f \in L^{\infty}(\R^d).
\end{align*}

The following regularity assumptions on the data of the problem will be used throughout
the paper
\begin{description}
\item[(A1)] Assume that the following three conditions hold:
\begin{enumeratei}
  \item  $b:\R^d \rightarrow \R^d$, $\sigma: \R^d\rightarrow \R^{d\times m}$, and
  $h:\R^d\rightarrow \R^l$  are bounded on
  $ \R^d$.    Moreover, $b$ is ${C}^1$ with bounded derivatives and
  $\sigma$ is ${C}^2$ with bounded first and second order derivatives.
  \item   There exists
  $\alpha>0$, such that $z^{\top}a(x)z\geq \alpha z^{\top}z$, $\forall x,z \in
  \R^d$.
  \item $\lambda: \R^d \rightarrow [\varpi_1,\varpi_2]$
is a continuous
 function for constants  $0<\varpi_1<\varpi_2 $.
\end{enumeratei}
\end{description}

\begin{example} \label{ex:credit-risk}  Filtering problems
with diffusive and point process observations arise naturally in  credit risk modeling.
This connection was studied systematically  in \citeN{bib:frey-runggaldier-10} and
\citeN{bib:frey-schmidt-10}, among others. In these papers reduced-form portfolio credit
risk models are considered where default times are doubly stochastic random times with
intensity driven by some economic factor process $X$. In a large homogeneous portfolio
the number of default events is thus given by some doubly stochastic Poisson process $N$
with intensity $\lambda(X_t)$. In line with reality, it is assumed that investors cannot
observe the  process $X$ directly, but  are confined to noisy observations of $X$,
modelled by a process $Z$ as in \eqref{eq:Z}. Moreover, they obviously observe the
occurrence of default events and hence the process $N$.

In this context the pricing of credit derivatives naturally leads to a filtering problem,
as we now explain. In abstract terms a credit derivative with maturity $T$ can be
described in terms of some $\cF_T^N$-measurable payoff $H$. Denote by $\Q$ the  risk
neutral measure used for pricing. The price of the credit derivative at time $t \le  T$
is then given by $H_t = \E^\Q( H \mid \cF^{Z,N}_t)$ (assuming zero interest rates for
simplicity). Using iterated conditional expectations we get
$$H_t = \E^\Q \Big ( \E^\Q (H \mid \cF_t ) \mid \cF^{Z,N}_t\Big ) \,.$$
It is well-known that the pair $(X, N)$ is an $\bbF$-Markov process. Hence for typical
claims $H$ one has the equality  $\E^\Q (H | \cF_t ) = h(t,X_t,N_t)$ for a suitable
function $h$, and we get that $ H_t = \E^\Q \big ( h(t,X_t, N_t) | \cF_t^{Z,N}\big)$.
The computation of this conditional expectation is a nonlinear filtering problem of the
type considered in the present paper.

For further information on incomplete-information models in credit risk we refer to the
to the survey article \citeN{bib:frey-schmidt-11}.

\end{example}

\subsection{The measure-valued Zakai equation}

The evolution equation for the measure $\pi_t(dx)$ is usually deduced using a
change of measure method. Define
\begin{align*}
{\Lambda}_t: =\prod_{\tau_n \leq t} { \lambda(X_{\tau_n -})}
\cdot \exp\bigg( \int_0^t h(X_s)^{\top}dW_s
  +\frac{1}{2}\int_0^t \|h(X_s)\|^2ds -\int_0^{t} (\lambda(X_{s})-1)ds
  \bigg)
\end{align*} for $t\in [0,T]$. Then the regularity assumptions in \textbf{(A1)}
imply that $(\Lambda^{-1}_t)_{t \in [0,T]}$ is a nonnegative martingale. We
define the  measure $\P^0$ by its Radon-Nikodym derivative
  $ d\P^0= {\Lambda}_T^{-1} d\P$.
The Girsanov theorem yields that, under $\P^0$,   $Z$ is  a standard Brownian
motion, that $N$ is  a Poisson process with intensity equal to one, and that
$X$, $Z$ and $N$ are independent. Denote by $ Y_t := N_t-t$ the compensated
Poisson process, such that under $\P^0$, $Y$ is a martingale. Then the
conditional distribution $\pi_t(dx)$ has a representation in terms of an
associated unnormalized version $\rho$:   denoting by $\E^0$ the expectation
w.r.t. $\P^0$, we obtain by the abstract Bayes rule   for any $f\in
L^{\infty}(\R^d)$
\begin{align}\label{eq:Bayes}
  \pi_t(f)=\frac{\E^0(f(X_t)\Lambda_t|\mathcal{F}^{{Z,N}}_t)}{\E^0(\Lambda_t|\mathcal{F}^{{Z,N}}_t)}
  =:\frac{\rho_t(f)}{\rho_t(1)}. \end{align}
It is well-known that the measure-valued process $\rho_t$ satisfies the \emph{classical
Zakai equation}: let $\rho_0(f):=\E[f(X_0)|\cF_0^{Z,N}]$. Then,  for any $f\in
C^2_b(\R^d)$, $t \in [0, T]$,
\begin{align} \label{eq:Zakaimv}
 \rho_{t}(f)= \rho_0(f) +\int_0^t \rho_{s}(\Gene f)ds
+\int_0^t \rho_{s}(fh^{\top})dZ_s  +\int_0^t
\rho_{s-}\Big(f(\lambda-1)\Big)dY_s,
\end{align}
$\P^0-a.s.$, see for instance Theorem~3.24 in \citeN{bib:bain-crisan-08} (only continuous
observations). A formal proof that under {\bf(A1)}, \eqref{eq:Zakaimv} holds in the setup
of the present paper is given in \citeN{bib:xu-thesis-10}, Theorem~2.9.

\subsection{The Zakai equation for the conditional density}

Our aim is to determine the dynamics of the Lebesgue-density of the unnormalized
conditional distribution $\rho_t(dx)$. Consider the separable Hilbert space $H
=L^2(\R^d)$ with norm $\|\cdot\|_H$ and scalar product $(\cdot,\cdot)$. To obtain
intuition, suppose that
$$ \rho_t(f) = (q_t,f) $$
for all $f\in C^2(\R^d)$ with compact support  and for some $H$-valued process $q =
(q_t)_{0 \le t \le T}$ such that $q_t(\cdot)$ can be identified with a smooth function.
Denote by the differential operator $\ccL^*$ the formal adjoint of the generator $\ccL$.
As $(q_t,\ccL f)=(\ccL^*q_t,f)$ the measure valued equation \eqref{eq:Zakaimv} simplifies
to
\begin{align}\label{eq:Zakai weak}
(q_t,f) &= (q_0,f) + \int_0^t (\ccL^* q_s,f) ds + \int_0^t (h^\top q_s,f) dZ_s + \int_0^t ((\lambda-1)q_{s-},f) dY_s.
\end{align}
This suggests that $q$ solves the stochastic partial differential equation (SPDE)
\begin{align*}
dq_t = \ccL^* q_t dt + h^\top q_t dZ_t + (\lambda-1)q_{t-} dY_t
\end{align*}
in an appropriate sense. The next step is to give this equation a precise mathematical
meaning using the theory for mild and weak solutions for SPDEs as in
\citeN{bib:peszat-zabcyk-07}. Besides the Hilbert space $H=L^2(\R^d)$ we consider the
Sobolev space $V=H^1(\R^d)\subset H$. We define an extension $\cA^*$ of $\ccL^*$ with
domain $D(\cA^*)\subset V$ as follows: $u\in V$ is an element of $D(\cA^*)$ if there
exists $f\in H$ such that for all $v \in V$
$$ -\frac{1}{2}\sum_{i,j=1}^d
\int_{\R^d} a_{ij}(x) \frac{\partial u}{\partial x_i} \frac{\partial v}{\partial x_j} \,
dx+\sum_{i=1}^d \int_{\R^d} \Big(b_i-\frac{1}{2}\sum_{j=1}^d  \frac{\partial
a_{ij}(x)}{\partial x_j} \Big) \frac{\partial v}{\partial x_i}u \,dx = (f,v),$$ and  we
set $\cA^*u=f$ in that case. If $u\in C_0^2(\R^d)$, we obtain that $f=\ccL^*u$ by
checking that $(f,v) = (u,\ccL v)$ with integration by parts. It is well-known that
$\cA^*$ generates an analytic $C_0$-semigroup $G^*$, see \citeN[Proposition A.10]{bib:da-prato-zabcyk-92}. (A $C_0$-semigroup $G^*$ is  a map from $[0,T]$ into $L(H,H)$
such that $G^*(0) = \text{id}$, $G^*(t+s) = G^*(t) G^*(s)$ and so that $G^*$ is continuous in the strong operator topology.)

\paragraph{Mild and weak solutions.}   Let $\Nspace$  denote the set of all
$\mathbb{F}^{Z,N}$-adapted, ${H}$-valued processes $\xi = (\xi_t)_{0\leq t \leq T}$,
continuous in the mean square norm, which are such that
\begin{align}\label{eq:norm}
| \xi|_T:=\bigg ( \sup_{t\in [0,T]}{\E}^0\Big (\|\xi(t) \|_{{H}}^2\Big ) \bigg)^{1/2}< \infty.
\end{align}
It is well-known that  $\Nspace $ is a Banach space with norm $|\cdot|_T$, see
\citeN{bib:germani-piccioni-87}.

Define the multiplication-operators $\cB \colon H \to H^l$, $\cB f := f h^\top$ and $\cC \colon H \to H$, $\cC f :=(\lambda-1)f$. A \emph{mild solution}  of the SPDE
\begin{align}\label{eq:zakai-density}
dq_t &= \cA^* q_t dt + \cB q_t dZ_t + \cC q_{t-} d Y_t.
\end{align}
 is a process $q \in \Nspace$ such that
\begin{align}\label{eq:zakai-Hilbert-mild}
q_t=G^*_t q_0+\int_0^t  G^*_{t-s} \cB q_{s}
dZ_s+\int_0^t G^*_{t-s} \cC q_{s-}\,
dY_s\,, \quad t \le T.
\end{align}
Denote by $\cA:=(\cA^*)^*$ the adjoint operator of $\cA^*$ and note that on $C^2_0(\R^d)$
the operator  $\cA$ coincides with the generator $\ccL$ of $X$. A \emph{weak solution} of the
SPDE \eqref{eq:zakai-density} is a process $q \in \Nspace$ such that for all $v \in
D(\cA)$
\begin{align}\label{eq:zakai-Hilbert-weak}
(q_t,v) &= (q_0,v) + \int_0^t (q_s, \cA v) \, ds + \int_0^t (q_s, \cB v) d Z_s + \int_0^t (q_{s-}, \cC v) d Y_s, \quad t \le T\,.
\end{align}
In our context $q$ is a weak solution of \eqref{eq:zakai-density} if and only if it is a
mild solution of that equation; this follows immediately from Theorem~9.15 in
\citeN{bib:peszat-zabcyk-07}.

\paragraph{The Zakai equation.} The following result describes the evolution of the density of the unnormalized
conditional distribution $\rho_t(dx)$.
\begin{theorem}\label{thm:density-zakai}   Assume that {\bf(A1)} holds.
Then for all $q_0 \in V$ there is a unique mild solution $q$  of the SPDE
\eqref{eq:zakai-density}. Moreover, $q_t \in H^1(\R^d)$ and for all $f \in L^2(\R^d)$ we
have that $ \rho_t(f) = (q_t,f)$.
\end{theorem}
In view of this result, equation \eqref{eq:zakai-density}  will be called the Zakai
equation for the unnormalized conditional density.

Theorem~\ref{thm:density-zakai} has been obtained  in \citeN{bib:pardoux-79} and in
\citeN{bib:germani-piccioni-87}  for the case of pure diffusion information and in
\citeN{bib:pardoux-79b} for the pure Poisson case ($h \equiv 0$). The extension to the
case of mixed observations may be found in \citeN{bib:xu-thesis-10}.

\subsection{The Galerkin approximation}\label{subsec:Galerkin approximation}

The Galerkin approximation for a (stochastic) PDE essentially projects the equation to a
finite-dimensional subspace.   In the case of  the  Zakai equation for the unnormalized
conditional density  the solution of the projected equation can be characterized in terms
of a finite-dimensional system of ordinary stochastic differential equations (SDEs), as
we now explain.

Formally the Galerkin approximation is defined as follows: Let ${\{e_1,e_2,\ldots
\}}\subset D(\mathcal{A}^*)\cap D(\cA)$ be a basis of the Hilbert-space $H$. Let $H_n$ be
the linear subspace  spanned by $\{e_1, \ldots, e_{n}\}$ and denote by $P_n$ the
projection from $H$ to $H_n$. We define the projection of the operator  $\cA^*$ by
$$(\mathcal{A}^*)^{(n)}:=P_n \mathcal{A}^* P_n\, ; $$
the operators $\cB^{(n)}$ and $\cC^{(n)}$ are defined analogously.
\begin{definition}\label{def:galerkin}The $n$-dimensional \emph{Galerkin approximation} of  \eqref{eq:zakai-density} is
the solution of
\begin{align}\label{eq:def-galerkin}
\begin{split}
dq_t^{(n)} &= (\cA^*)^{(n)} q_t^{(n)} dt + \cB^{(n)} q_t^{(n)} dZ_t + \cC^{(n)} q_{t-}^{(n)} dY_t, \\
q_0^{(n)} &= P_n q_0.
\end{split}
\end{align}
\end{definition}
As previously, there are two equivalent concepts of solutions. The mild solution of
\eqref{eq:def-galerkin} is obtained with $(G^*)^{(n)} := \exp(\cA^*)^{(n)}$. On the other
side, the weak form is obtained using the adjoint operator $\cA^{(n)} := ((\cA^*)^{(n)})^*$. Since for $u,v \in H$ one has $(P_n \cA^* P_n u,v) = (u,P_n
\cA P_n v)$ the weak form of the Galerkin approximation \eqref{eq:def-galerkin} becomes
\begin{align}\label{eq:Galerkin-approx-weak}
d(q_t^{(n)},v) &= (q_t^{(n)}, P_n \cA P_n v) dt + (q_t, P_n \cB P_n v) dZ_t + (q_{t-},P_n \cC P_n v) dY_t\,,\quad v \in H.
\end{align}

Note that for $v \in H_n^\bot$ we obtain that the differential  $d(q_t^{(n)},v)$ is equal to zero. Since  moreover $q_0^{(n)} =
P_n q_0 \in H_n$ it follows that $q_t^{(n)} \in H_n$ for $t \in [0,T]$ $\P$-a.s. Hence,
$q_t^{(n)}$ can be written as
\begin{align}\label{eq:approximation-qn}
q_t^{(n)} (x) = \sum_{i=1}^n \psi_i^{(n)} (t) e_i(x), \qquad t \in [0,T],
\end{align}
where $\psi_i^{(n)}$, $1 \le i \le n$   are called  \emph{Fourier coefficients}. Plugging
\eqref{eq:approximation-qn} into the weak form of the Galerkin
approximation~\eqref{eq:Galerkin-approx-weak}, we get that the Fourier coefficients
satisfy the following system of ordinary SDEs:
\begin{align*}
\sum_{i=1}^{n} (e_i, e_j) d\psi_i^{(n)}(t)&=
\Big(\sum_{i=1}^{n} \psi_i^{(n)}(t) (e_i, \cA e_j) \Big)dt
+\sum_{\ell = 1}^l \Big(\sum_{i=1}^{n} \psi_i^{(n)}(t) (e_i,h^{\ell} e_j)\Big)dZ^\ell_t\\
&+\Big(\sum_{i=1}^{n} \psi_i^{(n)}(t-) \Big(e_i,(\lambda-1)e_j\Big)\Big) dY_t.
\end{align*}
Define the $n \times n$ matrices $A,C, D $ and $B^\ell,\ell=1,\dots,l$ by their components:
\begin{align}\label{eq:GC-a}
 a_{ji}:= (e_i, \cA e_j), \, b^\ell_{ji}:=(e_i, h^\ell e_j),\,
c_{ji}:=(e_i,(\lambda-1) e_j), d_{ji}:=( e_i, e_j )\,.
\end{align}
As ${\{ e_1,e_2,\dots \}}$ is a basis of $H$, the matrix $D$
has full rank and is invertible. Using matrix notation we  obtain the following SDE
system for the vector-valued process  $\U^{(n)}:=(\psi^{(n)}_1,\ldots \psi^{(n)}_{n})^{\top}$,
\begin{equation}
\label{eq:Upsilon-e}
\begin{split}
d \Upsilon^{(n)}_t & =D^{-1}\Big( A\Upsilon^{(n)}_t dt + \sum_{\ell=1}^l B^\ell\Upsilon^{(n)}_t dZ^\ell_t
                     +  C \Upsilon^{(n)}_{t-}dY_t\Big), \\
  \Upsilon^{(n)}_0 &= D^{-1}{q}^{(n)}_0\,.
\end{split}
\end{equation}
This SDE system will be the starting point for our numerical analysis in
Section~\ref{sec:numerical-methods}. Note that for $\{e_1, e_2, \dots\}$  smooth, one has
$a_{ji} = ( e_i,  \Gene e_j )$ which is more convenient for computing the coefficients of
the system \eqref{eq:Upsilon-e}. For the case without point-process observation the SDE-system~\ref{eq:Upsilon-e} was already proposed by \citeN{bib:ahmed-radaideh-97}.

\paragraph{Moments of the conditional distribution.}
Obviously, the (normalized) conditional density of $\pi_t(dx)$  can be approximated via
\begin{align} \label{eq:aga-density}
p_{t}:=\frac{q_{t}}{\int_{\R^d} q_{t}(x)dx } \approx \frac{q_{t}^{(n)}}{\int_{\R^d} q_{t}^{(n)}(x)dx } =: p_t^{(n)} \,;
\end{align}
here $\approx$ means that we approximate the term on the left side by the Galerkin
approximation on the right side. In this case we have that $\E(f(X_{t})|\cF^{Z,N}_{t})
\approx (p_{t}^{(n)},f)$. On the other side, we can represent some characteristics of the
conditional distribution directly via $q_t$. Consider for simplicity the case $d=1$.
Denote by $\hat{x}_t$ and   $\hat \sigma_t^2$ be the conditional mean and  variance of the state
process at time $t\in[0,T]$. Then
\begin{align}\label{eq:CM-Galerkin}
\hat{x}_t=\E(X_t|\cF_t^{Z,N})=\frac{\int xq_t(x)dx}{\int
q_t(x)dx}\approx \frac{\int xq^{(n)}_t(x)dx}{\int q^{(n)}_t(x)dx}
=\frac{\sum_{i=1}^{n}\psi_i^{(n)}(t)(x,e_i) }
  {\sum_{i=1}^{n}\psi_i^{(n)}(t)({1},e_i)}.
\end{align}
Note that the second equality follows from the definition of the unnormalized
distribution, see  \eqref{eq:Bayes}. For the last equality we used
\eqref{eq:approximation-qn}. In a similar way we approximate in $\hat \sigma_t^2=
\E\big((X_t-\hat{x}_t)^2|\cF_t^{Z,N}\big) = \E(X_t^2|\cF_t^{Z,N})-(\hat{x}_t)^2$ the
conditional second moment by
\begin{align}\label{eq:CV-Galerkin}
\E(X_t^2|\cF_t^{Z,N})
\approx \frac{\sum_{i=1}^{n}\psi_i^{(n)}(t)(x^2,e_i) }
  {\sum_{i=1}^{n}\psi_i^{(n)}(t)({1},e_i)}.
\end{align}
Analogously all  moments of the conditional distribution can be represented by the
Fourier coefficients. Notice that $({1},e_i) $, $(x,e_i)$ and $(x^2,e_i)$ are independent
of the observation and can be computed off-line (we implicitly assume that these
integrals exist for the chosen basis functions).

\section{Convergence results} \label{sec:convergence}

This section gives sufficient conditions for the convergence of the Galerkin
approximation $q^{(n)}$ defined in \eqref{eq:approximation-qn} to the solution
of the Zakai equation $q$ from \eqref{eq:zakai-density} in an appropriate
sense. The following theorem is the main theoretical result of the paper:

\begin{theorem}\label{Thm:convergence_linear_part}Assume that {\bf(A1)} holds.
Let $q$ be the  solution of the Zakai equation in
\eqref{eq:zakai-density} and $q^{(n)}$ be the
corresponding Galerkin approximation.  Then,  for any $q_0\in V$,
\begin{align*}
\sup_{t\in[0,T]}\E^0 (\|q^{(n)}_t-q_t\|_H^2)\rightarrow 0, \quad
\text{as}\quad n\rightarrow \infty,
\end{align*}
  if and only if,  for any $x\in H$,
  \begin{align}\label{eq:conditions_convergence}
  \lim_{n\rightarrow
\infty}\sup_{t\in [0,T]} \Big\| \big(\exp(P_n \mathcal{A}^* P_n t)-G_t^*\big)x
\Big\|_H=0.
\end{align}
\end{theorem}
Note that $G_t^*x$ is the solution of the Kolmogorov forward PDE with initial
condition $x$ (the PDE describing the evolution of the transition density of
$X$) and $\exp(P_n \mathcal{A}^* P_n t) x$ is the Galerkin approximation to
this (deterministic) PDE. Hence Theorem~\ref{Thm:convergence_linear_part}
shows that the Galerkin approximation for the Zakai equation converges if and
only if the Galerkin approximation for the deterministic forward equation
converges.

Necessary and sufficient conditions for \eqref{eq:conditions_convergence} to hold can be
obtained by means of the Trotter-Kato theorem. A convenient condition that ensures
\eqref{eq:conditions_convergence} under \textbf{(A1)} is that
\begin{align}\label{eq:convenientcondition}
\bigcup_{n \in \N} H_n \text{  is dense in } V;
\end{align}
see Theorem 4, \citeN{bib:germani-piccioni-84}.

\subsection*{Proof of Theorem \ref{Thm:convergence_linear_part}}

The remainder of this section is devoted to the proof  of Theorem
\ref{Thm:convergence_linear_part}. The essential part of the proof is a continuity result
for the mild form of the  Zakai equation, see Proposition \ref{thm:Continuity Theorem}
below. This result is an extension of a result from \citeN{bib:germani-piccioni-87} where
the case of  continuous observation is treated. We recall the mild form of the Zakai
equation in the Banach  space $\Nspace $, $ q_t=G^*_t q_0+\int_0^t G^*_{t-s} \cB q_{s}
dZ_s+\int_0^t G^*_{t-s} \cC q_{s-}, dY_s $, $t \le T$ where for $f \in H$, $\cB f =
h^\top f$ and $\cC f=(\lambda-1)f$.

We start  by introducing some necessary operator spaces. By $\mathcal{S}$ we
denote the space of all $C_0$-semigroups of linear bounded operators  from $H$
to $H$ such that there exists $\bar{S} \in \R^+$ with
 \begin{align}\label{eq:bounded-S}
\sup_{t\in [0,T]}\|S_t \|\leq \bar{S} \; \text{ for all $S\in \mathcal{S}$}.
\end{align}
We endow $\mathcal{S}$ with the topology of  uniform strong convergence on
$[0,T]$, i.e.~a sequence  $(S^{(n)})$ in $\mathcal{S}$ converges to $S\in
\mathcal{S}$ if for all $x\in H$
\begin{align*}
\lim_{ n \rightarrow \infty}\sup_{t\in [0,T]} \Big\|
(S^{(n)}_t-S_t)x \Big\|_H=0.
\end{align*}

For any $l \in \N$ denote by $\mathcal{U}^l$   the space of linear bounded
operators from $H$ to $H^l$ ($l$-fold product of $H$). In the special case
$l=1$ we  write $\mathcal{U}=\mathcal{U}^1$. An operator $A \in \mathcal{U}^l$
can be written  component-wise: for all $x \in H$,
\begin{align*}
A x=(A^{1}x,\ldots,A^{l}x)^\top
\end{align*}
with $A^{i} \in \mathcal{U}$. The space $\mathcal{U}^l$ is endowed with the strong topology, that is
a sequence $(A^{(n)})$ in $\mathcal{U}^l$ converges to $A\in \mathcal{U}^l$, if for all $x\in H$
\begin{align*}
\lim_{n \rightarrow \infty}  \Big\| (A^{(n)}-A) x \Big\|_{H^l}=0.
\end{align*}

\paragraph{The studied SPDEs.} For the proof we study a more general class of
linear stochastic partial differential equations that includes  the Zakai
equation \eqref{eq:zakai-Hilbert-mild} as a special case. Consider a generic
semigroup $S \in \mathcal{S}$ and generic linear operators $B \in \mathcal{U}^l$,
$C\in \mathcal{U}$ and some  $f \in H$. In the sequel we study the following
equation in $\Nspace$:
\begin{align}\label{eq:zakai-Hilbert} \xi_t=S_t f+\int_0^t
S_{t-s} B \xi_{s} d Z_s +\int_0^t S_{t-s} C
\xi_{s-} dY_s, \quad t\in[0,T].
\end{align}

The following decomposition  of this equation  is the starting point  for our
analysis: define the   linear operator $L$ on $\Nspace$  by
\begin{align}\label{eq:conti-L}
(L\xi)(t):=\int_0^t
S_{t-s} B \xi_{s} d Z_s +\int_0^t S_{t-s} C
\xi_{s-} dY_s
\end{align}
for all $t \in [0,T]$ and $\xi \in H$. Furthermore, set  $ \xi^{[0]}_t:=S_t f$
such that $\xi^{[0]}\in \Nspace$.
We obtain that  \eqref{eq:zakai-Hilbert} can be
rewritten as the following equation in $\Nspace $
\begin{align}\label{eq:zakai-Hilbert_abstract}
\xi=\xi^{[0]}+L\xi.
\end{align}
The operator $L$ is a bounded linear operator and it is moreover
quasinilpotent, as the following estimate shows.
\begin{lemma}\label{lem:Lbounds}
Set  $\gamma:= \sqrt{T}\bar{S}(\|B\|^2+\|C\|^2)^{\frac{1}{2}}$. Then, for all  $n\in\N$
\begin{align}\label{eq:norm_Ln}
\|L^{n}\|^{\frac{1}{n}} \leq\frac{\gamma}{(n!)^{\frac{1}{2n}}}.
\end{align}
\end{lemma}
The proof is given in  Appendix~\ref{app:Proofs}.

\begin{lemma}\label{thm:zakai-abstract-eu}  Equation
\eqref{eq:zakai-Hilbert_abstract} has a unique solution in
$\Nspace$,
\begin{align}\label{eq:zakai-Hilbert_abstract_solution}
\xi=(I-L)^{-1}\xi^{[0]}:=\sum_{i=0}^{\infty} L^i \xi^{[0]},
\end{align}
and $(I-L)^{-1}:\Nspace\rightarrow \Nspace$ is a bounded linear operator:
$\|(I-L)^{-1}\|<\kappa$ with $\kappa=\frac{2}{\sqrt{3}} e^{2 \gamma^2}$.
\end{lemma}
\begin{proof}The crucial part in the proof of the lemma is the estimate
\begin{align*}
\sum_{n=0}^{\infty} \| L \|^n  &\le \sum_{n=0}^{\infty}
     \frac{\gamma^n}{(n!)^{\frac{1}{2}}}
     = \sum_{n=0}^{\infty} 2^{-n} \frac{ (2\gamma)^n}{(n!)^{\frac{1}{2}}}
\leq \left ( \Big(\sum_{n=0}^{\infty}2^{-2n} \Big) \Big ( \sum_{n=0}^{\infty}
\frac{ (2 \gamma)^{2 n}}{n !} \Big ) \right )^{\frac{1}{2}} = \kappa,
\end{align*}
which shows that the Volterra series $\sum_{i=0}^n L^i$ does in fact converge as $ n
\to\infty$.
\end{proof}

In view of Lemma~\ref{thm:zakai-abstract-eu}  we can define the mapping ${F}:
H\times \mathcal{U}^l \times \mathcal{U} \times \mathcal{S} \rightarrow
\Nspace$  by
$$ F(f,B,C,S) := \xi,$$
where $\xi$ is   the unique solution in $\Nspace$ of \eqref{eq:zakai-Hilbert} with coefficients $(f,B,C,S)$.
The following result shows that $F$  is continuous.
\begin{proposition}\label{thm:Continuity Theorem}
Consider sequences $(f^{(n)})$, $(B^{(n)})$, $(C^{(n)})$ and $(S^{(n)})$ in $H$, $\mathcal{U}^l$, $\mathcal{U}$ and $\mathcal{S}$, converging
to $f\in H$, $B\in \mathcal{U}^l$, $C\in \mathcal{U}$ and $S \in \mathcal{S}$, respectively.
Then,
\begin{align*}
\Big|  F(f^{(n)},\, B^{(n)} \,C^{(n)},\, S^{(n)})
  -F(f, B, C, S) \Big|_T \rightarrow 0, \quad \text{as} \quad
  n\rightarrow \infty.
\end{align*}
\end{proposition}
\begin{proof}[Proof of Proposition~\ref{thm:Continuity Theorem}]

Since $S^{(n)} \rightarrow S$ , $B^{(n)} \rightarrow B$ and $C^{(n)}
\rightarrow C$, by the uniform boundedness principle there exist $\bar{N}$ and
a constant $\bar{\gamma}$ such that
\begin{align}
\sup_{t\in [0,T], n\geq \bar{N}}\Big\{\| S_t \| \vee \| S^{(n)}_t\|
\vee \|B\| \vee \|B^{(n)} \| \vee \|C\|\vee \|C^{(n)} \|
\Big\}\leq \bar{\gamma}.
\end{align}
In the following, we only consider sufficiently large $n>\bar{N}$. Set
\begin{align*}
\xi:=F(f, B,C, S), \quad \xi^{(n)}:=F(f^{(n)}, B^{(n)}, C^{(n)}, S^{(n)}).
\end{align*}
Together with $\xi_{t}^{[0,(n)]}:=S^{(n)}_t f^{(n)}$
we define $L^{(n)}$ by
\begin{align*}
(L^{(n)} \xi)(t)&:=\int_0^t S^{(n)}_{t-s} B^{(n)}
(\xi_{s})dZ_s +\int_0^t S^{(n)}_{t-s} C^{(n)}(\xi_{s-})dY_s
\end{align*}
for all $\xi\in\Nspace$. Then, by the very definition of $F$,
\begin{align*}
\xi^{(n)}=\xi^{[0,(n)]} +L^{(n)}  \xi^{(n)}, \quad
\xi=\xi^{[0]}+L{\xi}.
\end{align*}
Hence
\begin{align}
\xi^{(n)}-\xi
&=(\xi^{[0,(n)]}-\xi^{[0]})+L^{(n)}(\xi^{(n)}-\xi)+(L^{(n)}-L)\xi \nonumber \\
\label{eq:chi-diff}
&=(I-L^{(n)})^{-1} \Big((\xi^{[0,(n)]}-\xi^{[0]})+(L^{(n)}-L)\xi \Big).
\end{align}
By Lemma \ref{thm:zakai-abstract-eu} there exists a
constant $\kappa=\kappa({\bar{\gamma}})$, such that
\begin{align}\label{eq:bounded-1-L}
\big\| (I-L^{(n)})^{-1} \big\| \leq  \kappa.
\end{align}
Furthermore, as $S^{(n)} \in \mathcal{S}$,
\begin{align*}
\nonumber \big| \xi^{[0,(n)]} -\xi^{[0]} \big|_T^2 =& \sup_{t\in[0,T]} \Big\| S^{(n)}_tf^{(n)}-S_tf \Big\|_H^2\\
\nonumber \le &2\sup_{t\in [0,T]} \Big\| S^{(n)}_t(f^{(n)}-f)
\Big\|_H^2
+2 \sup_{t\in[0,T] } \Big\| (S^{(n)}_t-S_t)f\Big\|_H^2 \\
\leq & 2 \bar{\gamma}^2 \| f^{(n)}-f
\|_H^2 +2 \sup_{t \in [0,T] }\Big\| (S^{(n)}_t-S_t)f \Big\|_H^2.
\end{align*}
The last term converges to zero as  $(f^{(n)})$ and $(S^{(n)})$ converge to  $f$ and $S$, respectively.

Finally, we show  that $(L^{(n)}-L)\xi$ converges to zero. From the definition
of $L$ and $L^{(n)}$ we obtain by the It\^o-isometry  that
\begin{align*}
\lefteqn{ | (L^{(n)}-L)\xi |_T^2 =
\sup_{t\in[0,T]} \E^0 \Big(\|((L^{(n)}-L) \xi) (t)\|_H^2 \Big)}\\
 & = \sup_{t\in[0,T]} \E^0  \bigg(  \int_0^t \Big\|(
S^{(n)}_{t-s}B^{(n)} -S_{t-s}B ) \xi_{s}\Big\|_{H^l}^2
ds   + \int_0^t \Big\|( S^{(n)}_{t-s}C^{(n)}-S_{t-s}C )
\xi_{s}\Big\|_H^2   ds \Big) \\
&\leq    2 \bigg[\sup_{t\in[0,T]} \E^0  \Big(   \int_0^t
\Big\|S^{(n)}_{t-s}(B^{(n)}-B) \xi_{s}\Big\|_{H^l}^2 ds\Big)
+  \sup_{t\in[0,T]} \E^0\Big( \int_0^t \Big\|(S^{(n)}_{t-s} -S_{t-s})B  \xi_{s}\Big\|_{H^l}^2 ds\Big)\\
&+   \sup_{t\in[0,T]} \E^0 \Big(\int_0^t \Big\| S^{(n)}_{t-s}(C^{(n)}-C)
\xi_{s}\Big\|_H^2   ds \Big) +  \sup_{t\in[0,T]} \E^0 \Big(\int_0^t \Big\|
 (S^{(n)}_{t-s}-S_{t-s})C  \xi_{s}\Big\|_H^2   ds \Big)\bigg]\\
&:=2(E_1+E_2+E_3+E_4).
\end{align*}
We consider the terms $E_1$ to $E_4$ separately. Observe that by \eqref{eq:bounded-S},
\begin{align*}
E_1
  \leq
\bar{\gamma}^2  \E^0  \Big(\int_0^T \Big\|
 \Big(B^{(n)} - B\Big)  ( \xi_{\tau})\Big\|_{H^l}^2 d\tau \Big).
\end{align*}
As  $B^{(n)}$ converges to $B$, we have for all  $t\in [0,T]$ and $\omega\in
\Omega$
 \begin{align*}
\Big\|
 \Big(B^{(n)} - B\Big)  ( \xi_{s}(\omega))\Big\|_{H^l}^2\rightarrow 0.
\end{align*}
In order to show that  $E_1 \to 0$ as $n\to \infty$ we  apply dominated
convergence. Since   $\|B^{(n)}\|, \|B \| \leq \bar\gamma$ we get
$$\Big\| (B^{(n)} -B\Big) ( \xi_{s})\Big\|_{H^l}^2 \leq 4 \bar{\gamma}^2 \|
   \xi_{s}\|_{H}^2$$
and the last term is integrable since $\E^0  \Big(\int_0^T \|  \xi_{s}\|_{H}^2
ds\Big) \leq   T  |  \xi |_T^2< \infty.$

In a similar way
\begin{align}
E_2
 \leq &\sup_{t\in[0,T]} \E^0\Big( \int_0^t \sup_{s\leq  \tau \leq T}
 \Big\|( S^{(n)}_{\tau-s} -S_{\tau-s}) B \xi_{s}\Big\|_{H^l}^2 ds \Big)\nonumber\\
  = &  \E^0\Big( \int_0^T \sup_{s\leq  \tau \leq T}
 \Big\|( S^{(n)}_{\tau-s} -S_{\tau-s}) B \xi_{s}\Big\|_{H^l}^2 ds \Big)
\end{align}
while  uniform strong convergence of $S^{(n)}$ gives
$\sup_{s\leq  \tau \leq T}
 \Big\|( S^{(n)}_{\tau-s} -S_{\tau-s}) B \xi_{s}(\omega)\Big\|_{H^l}^2\rightarrow  0$
 for all $\omega \in \Omega$. As
\begin{align*}
\E^0\Big( \int_0^T \sup_{s\leq  \tau \leq T}
 \Big\|( S^{(n)}_{\tau-s} -S_{\tau-s}) B \xi_{s}\Big\|_{H^l}^2 ds \Big) \leq
 4 \bar{\gamma}^4 \E^0  \Big( \int_0^T \|\xi_{s}\|_H^2 ds\Big) \leq  4 \bar{\gamma}^4  T  |  \xi |_T^2
\end{align*} and  $ |  \xi |_T<\infty $ by Lemma \ref{thm:zakai-abstract-eu} we obtain
again by dominated convergence that $E_2 \to 0$. Analogously we obtain $E_3\to 0$ and $E_4\to 0$
and we conclude.
\end{proof}

Finally we turn to the
\begin{proof}[Proof of Theorem \ref{Thm:convergence_linear_part}]
Under Condition~\eqref{eq:conditions_convergence} the assumptions of Proposition
\ref{thm:Continuity Theorem} are clearly satisfied for the Galerkin approximation of the
Zakai equation, as $P_n x\rightarrow x$, $P_n B P_n x \rightarrow Bx$ and $P_n C P_n x
\rightarrow C x$ for all $x \in H$. For the proof of the converse statement (the fact
that \eqref{eq:conditions_convergence} is also necessary for the convergence of the
Galerkin approximation) we refer to the proof of Theorem~6.1 in
\citeN{bib:germani-piccioni-87}.
\end{proof}

%
%
%

\section{Numerical methods}

\label{sec:numerical-methods}

In this section
we discuss various aspects of the practical implementation of the Galerkin approximation
for the Zakai equation. We begin with a few algorithms for the numerical solution of the
SDE system \eqref{eq:Upsilon-e}. In Section~\ref{Hermite Expansion} we consider the
special class of basis functions constructed from Hermite polynomials. In
Section~\ref{section:The Adaptive  Galerkin Approximation} we finally show that 
the efficiency of the Galerkin approximation can be improved substantially if the scale
and the location of the bases are changed adaptively.

\subsection{Numerical solution of the Zakai equation}
\label{subsec:numerical-solutions}

In order to solve  the SDE system in \eqref{eq:Upsilon-e} numerically, we discretize the
system in time. As numerical schemes we consider the Euler-Maruyama method and the
splitting-up method. While the Euler-Maruyama method is fast to implement, it can become
quite unstable  if the time step  is relatively large (see Figure
\ref{fig:estimation-linear-KBt1}). This difficulty can be overcome with the splitting-up
method. Note that in practical filtering problems the observation often comes at discrete
time points, so that the time-discretization step can not be chosen arbitrarily small.

Our aim is to approximate Equation \eqref{eq:Upsilon-e}. It will be convenient to use the process $N$ as driver (instead of $Y=N-t$). Rewriting  Equation \eqref{eq:Upsilon-e} leads to
\begin{align}\label{eq:16-2}
d \Upsilon^{(n)}_t  =D^{-1} \Big(( A - C ) \Upsilon^{(n)}_t dt +  \sum_{\ell =1}^l B^\ell \Upsilon^{(n)}_t dZ_t
                     +  C \Upsilon^{(n)}_{t-}dN_t\Big ), \quad  \Upsilon^{(n)}_0 =q^{(n)}_0.
\end{align}
Consider some  the equidistant partition $0 =t_0<t_1<\cdots <t_K= T$  with step size
$\Delta:=T/K$. In the sequel we consider $n$ and the partition fixed and denote the approximation at the  time point $t_k$, $0 \le k \le K$,  by $\U_k = (\psi_{k,1}, \dots, \psi_{k,n})^\prime$.

\paragraph{Euler-Maruyama method.}
The Euler-Maruyama method  (EM method) generalizes the Euler method to stochastic differential
equations, see e.g.~\citeN{bib:McLachlan}.  It is described in the following algorithm:
\begin{algorithm}[EM method]\label{algo:EM}
For  $k=1,\dots,K$,   compute $\U_{k}$ from   $\U_{k-1}$ by
\begin{align*}
 \U_{k}= &\U_{k-1}+D^{-1}\Big(
  (A -C)  \U_{k-1} \Delta +   \sum_{\ell=1}^l B^\ell  \Upsilon_{k-1} (Z^\ell_{t_{k}}-Z^\ell_{t_{k-1}}) +  C \U_{k-1}
(N_{t_{k}}-N_{t_{k-1}})\Big)\,.
\end{align*}
\end{algorithm}

\paragraph{Splitting-up method.}
The splitting-up method (SU method) is a numerical method  based on semigroup theory. It
decomposes the original SDE into stochastic and a deterministic equations which are
easier to handle. We refer to  \citeN{bib:bensoussan-et-al-90} and \citeN{bib:leGland-92}
for further details in the  case of continuous observations. Here we propose an extension
of the method to the case with mixed observations. For simplicity, we assume that the basis
$\{e_1,e_2,\dots,e_n\}$ consists of orthonormal functions so that  $D=D^{-1}=I_{n}$.

Intuitively, the SU method computes $\U_k$ from $\U_{k-1}$ in three steps:
the first step uses only the $dt$-part of equation \eqref{eq:16-2} and returns the solution
of the SDE $d\U^1_t = ( A - C ) \Upsilon^{1}_t dt$. The solution of this  equation on $[t_{k-1},t_k]$ is the matrix
exponential $\Upsilon^1_{t_k}=\exp\big((A-C)\Delta\big)\Upsilon^1_{t_{k-1}}$. Step 2
incorporates the new information from $Z$ via the linear SDE $ d
\Upsilon^2_t= B \Upsilon^2_t dZ_t$ with initial condition $\Upsilon^2_{t_{k-1}} =
\Upsilon^1_{t_{k}}$. The solution of this SDE is given by the matrix exponential
$$ \Upsilon^2_{t_{k}}  =\exp\bigg(\sum_{\ell=1}^l \Big(B^\ell (Z_{t_k}-Z_{t_{k-1}})-\frac{1}{2}(B^\ell)^2\Delta\Big)\bigg)\Upsilon^1_{t_{k}}.$$
The new jump  information is incorporated via the linear equation $ d\Upsilon^3_t= C \Upsilon^3_{t-} dN_t$, this time with initial condition
$\Upsilon^3_{t_{k-1}} = \Upsilon^2_{t_{k}}$, which gives
$$ \Upsilon^3_{t_{k}}=(I_n+C)^{(N_{t_k}-N_{t_{k-1}})} \Upsilon^2_{t_{k}}. $$
These steps lead to the following algorithm:
\begin{algorithm}[SU method]\label{algo:splitting-up}
For  $k=1,\dots,K$, compute  $\U_{k}$ from   $\U_{k-1}$ by
 \begin{itemize}
  \item[(1)]  Compute  $\Upsilon^1_{k}:=\exp\big( (A-C)\Delta\big)\Upsilon_{k-1}$.
  \item[(2)] Compute  $ \Upsilon^2_{k}:=\exp\big(\sum_{\ell=1}^l (B^\ell (Z_{t_k}-Z_{t_{k-1}})-\frac{1}{2}(B^\ell)^2\Delta)\big)\Upsilon^1_{k}$.
\item[(3)] Return $ \U_{k}:=(I_n+C)^{(N_{t_k}-N_{t_{k-1}})} \Upsilon^2_{k}$
\end{itemize}
\end{algorithm}

%
%

 \subsection{Galerkin approximation based on Hermite polynomials}\label{Hermite Expansion}

The choice of the  basis functions has a large impact on the quality of the Galerkin
approximation. \citeN{bib:ahmed-radaideh-97} propose to use Gaussian  series, i.e.~a
series build by densities of $n$-dimensional Gaussian distributions with different means
and arbitrary positive, symmetric covariance matrices.  It is shown 
that these are linearly independent and complete and hence they can be used to
construct Galerkin approximations as described above.

In this paper we instead consider a basis computed from Hermite polynomials. This basis
has a number of computational advantages over  Gaussian series   as will become clear
below. We start by recalling some properties of Hermite polynomials, see e.g.\
\citeN{bib:courant-hilbert-68}. The $i$th Hermite polynomial is defined  by
\begin{align}\label{eq:def-hermite}
f_i(x)=(-1)^i e^{x^2/2} \frac{d^i}{d x^i}e^{-x^2/2}, \quad x\in \R
\end{align}
with $i=0,1,2,\ldots$. It follows that $ f_0(x)=1,$ $ f_1(x)=x,$
 $f_2(x)=x^2-1$ and $f_3(x)=x^3-3x$. The Hermite  polynomials  are orthogonal with respect to the weighting
function $\phi(x):=(2\pi)^{-\nicefrac{1}{2}}\, e^{-\nicefrac{x^2}{2}}$, as $\int_{\R} f_i(x) f_j(x)
\phi(x)dx=i! \, \ind{i=j}$. Consequently, the functions $ e_1, e_2, \ldots$ given by
\begin{align}\label{eq:def-hermite-basis}
e_i(x):=\sqrt{\frac{\phi(x)}{(i-1)!}} f_{i-1}(x), \quad x\in \R,
\end{align}
constitute an orthonormal basis of $L^2(\R)$, which we call \emph{Hermite basis}.
In the following result we deduce  the convergence of the Galerkin approximation with
the use of  \eqref{eq:convenientcondition}.

\begin{proposition} \label{cor7.3}
Assume that {\bf (A1)} holds and that $q_t^{(n)}$ is the Galerkin approximation of the Zakai equation with respect to the Hermite basis. Then,  for any $q_0\in V$,
\begin{align*}
\sup_{t\in[0,T]}\E^0 (\|q^{(n)}_t-q_t\|_H^2)\rightarrow 0, \quad
\text{as}\quad n\rightarrow \infty.
\end{align*}
\end{proposition}
The proof is given in Appendix~\ref{app:Proofs}. To actually obtain the Galerkin approximation under the Hermite basis one computes the
coefficient matrices  $A, B^1,\dots,B^l, C$ as in \eqref{eq:GC-a} with respect to the
Hermite basis (the matrix $D$ is the identity matrix as the Hermite basis consists of orthonormal functions) and then solves \eqref{eq:Upsilon-e} numerically by one of the methods
described before. In some special cases it is even possible to  obtain  explicit formulas for the entries of the coefficient matrices, as is illustrated in the following example.

\begin{example}[Kalman filter with point process observations]\label{exam:linear_model}
Consider  $d=m=l=1$ and assume that $b(x)=bx$, that  $\sigma(x)=\sigma$ and that $X_0 \sim \cN(\mu_0,\sigma_0^2)$, so that
$$ X_t = X_0 + \int_0^t b X_s ds + \sigma V_t$$
is a linear Gaussian process with generator $\ccL f (x)= bx f^\prime(x)+\frac{\sigma^2}{2} f^{\prime\prime}(x).$  Assume moreover that $h$ and $\lambda$ are of the form  $h(x)=hx$ and $\lambda(x)=\lambda x^2$ with $\lambda >0$ such that
the observation is given by $Z_t = \int_0^t h X_s ds + W_t$ and by the doubly stochastic Poisson process $N$ with intensity $(\lambda X_t^2)_{t \ge 0}$. In the next lemma we give explicit formulas for  the coefficient matrices $A$, $B$ and $C$.

\begin{lemma}\label{lemma:coefficients-kalman}
Consider  Then with $h(x)=hx$ and $\lambda(x)=\lambda x^2$ we obtain
\begin{align}\label{eq:kalman-a}
 a_{ji}=(e_i,\ccL e_j) &=
\begin{cases}
- \frac{b}{2} + \frac{\sigma^2}{8}(1-2i) & \text{ for } i=j \\
(-\frac{b}{2} + \frac{\sigma^2}{8}) \sqrt{j(j+1)} & \text{ for } i=j+2 \\
(\frac{b}{2} + \frac{\sigma^2}{8}) \sqrt{(j-1)(j-2)}& \text{ i=j-2}
\end{cases} \\ \label{eq:kalman-b}
 b_{ji}=(e_i, h(\cdot) \,  e_j  ) &=
\begin{cases}
h \,  \sqrt{j} & \text{ for } i=j+1 \\
h \, \sqrt{j-1} & \text{ for } i=j-1,
\end{cases} \\ \label{eq:kalman-c}
c_{ji} = (e_i, (\lambda(\cdot) -1) e_j) &= \begin{cases}
\lambda \, (2j-1) -1 & \text{ for } i=j \\
\lambda \, \sqrt{j(j+1)} & \text{ for } i=j+2 \\
\lambda \, \sqrt{(j-1)(j-2)} & \text{ for } i = j-2,
\end{cases}
\end{align}
and zero for all other cases.
\end{lemma}
We give the proof in the Appendix~\ref{app:Proofs}. In order to set up the Galerkin approximation one moreover needs to project $q_{0}$ on the subspace $H_n$  generated by the Hermite basis. This is done with some additional notation in Lemma~\ref{lemma:projection-of-initial-density} below.

We mention that explicit computation of $ a_{ji}=(e_i,\ccL e_j)$ is possible also for other state processes with different generator such as a  CIR process.
\end{example}

\paragraph{Computation of moments.}
Recall from \eqref{eq:CM-Galerkin} that in order to compute mean and variance of the
filter distribution via Galerkin approximation one needs to determine be integrals
$(x^j,e_i)$. For the Hermite basis this can be done analytically. In order to present the corresponding formulae we introduce some additional notation.
The $i$-th Hermite function is an polynomial of order $i$, and we  denote by
$\vartheta^i_0, \dots,\vartheta^i_i$  the coefficients in the representation $f_i(x)=\sum_{k=0}^i \vartheta^i_k x^k.$
Conversely,  any power of $x$ can be represented as linear combination of Hermite
polynomials and we write
  $x^i = \sum_{k=0}^i \iota^i_k f_k(x).$
Now we have
\begin{lemma} \label{lemma:hermite-moments} For the Hermite basis it holds that
\begin{align}\label{eq:Galerkin-hermite-r}
 ( x^j, e_i)  = \frac{\sqrt{2}(2\pi)^{\frac{1}{4}}}{\sqrt{(i-1)!}}\sum_{k=0}^{i-1} \vartheta^{i-1}_k  2^{\frac{k}{2}}  \iota^{k+j}_0, \quad j= 0,1,\dots
 \end{align}\label{lem:4.6}
\end{lemma}
The proof of this lemma is given in Appendix~\ref{app:Proofs}.

\begin{example}[Example~\ref{exam:linear_model} continued] Using the above notation we may also  compute the projection of the initial density for the case of the Kalman filter with point process observation:
 \begin{lemma}\label{lemma:projection-of-initial-density}
Set $a:= \frac{2\mu_0}{2+\sigma^2}$, $b:= \sqrt{\frac{2\sigma^2}{2+\sigma^2}}$ and $d:= \frac{4 \mu_0^2}{8 \sigma^4}- \frac{\mu_0}{2\sigma^2}.$ Then
\begin{align*}
    ( q_0, e_j ) =&\frac{1}{\sqrt{ \sigma_0^2 \, (j-1)!}} \frac{e^d}{(2\pi)^{-1/4}} \sum_{m=0}^{j-1} \sum_{k=m}^{j-1}  \vartheta^{j-1}_k \, {k \choose m} \, a^{k-m} \, b^m \iota_0^m.
\end{align*}
\end{lemma}
We give the proof in  Appendix~\ref{app:Proofs}.
\end{example}


%

\subsection{The adaptive  Galerkin approximation}\label{section:The Adaptive  Galerkin Approximation}

During the filtering process the conditional distribution $\pi_t(dx)$ typically changes
location and scale. This can create problems for the Galerkin approximation with a fixed
basis. For instance, the graphs in Figure \ref{fig:error-v} show that while the standard
Hermite polynomials do approximate the density of a normal distribution well if the mean
is close to zero and if the variance $\sigma^2$ lies between one and four, the fit becomes
substantially worse if $\mu$ is substantially different from zero or if $\sigma^2$ is
 outside of the interval $[1,4]$. Hence we propose an adaptive scheme, called \emph{adaptive
Galerkin approximation} (AGA), which improves the numerical performance of the Galerkin approach significantly.

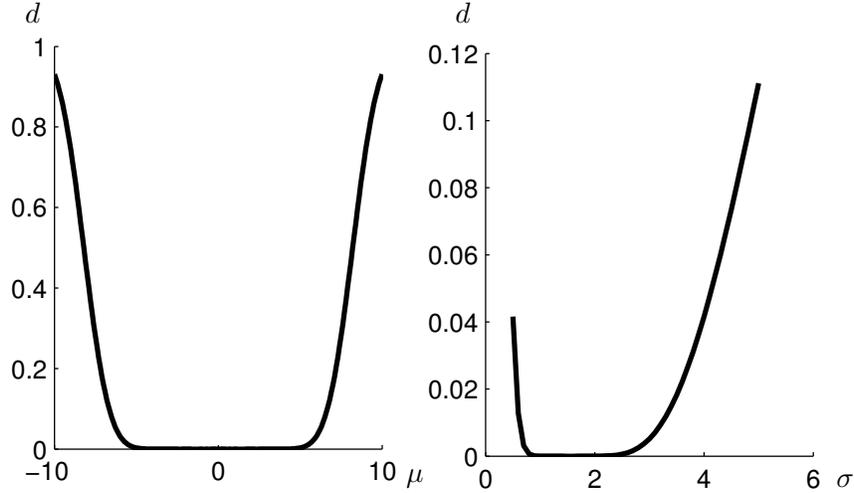
\begin{figure}
\begin{center}
    \begin{overpic}[width=13cm]{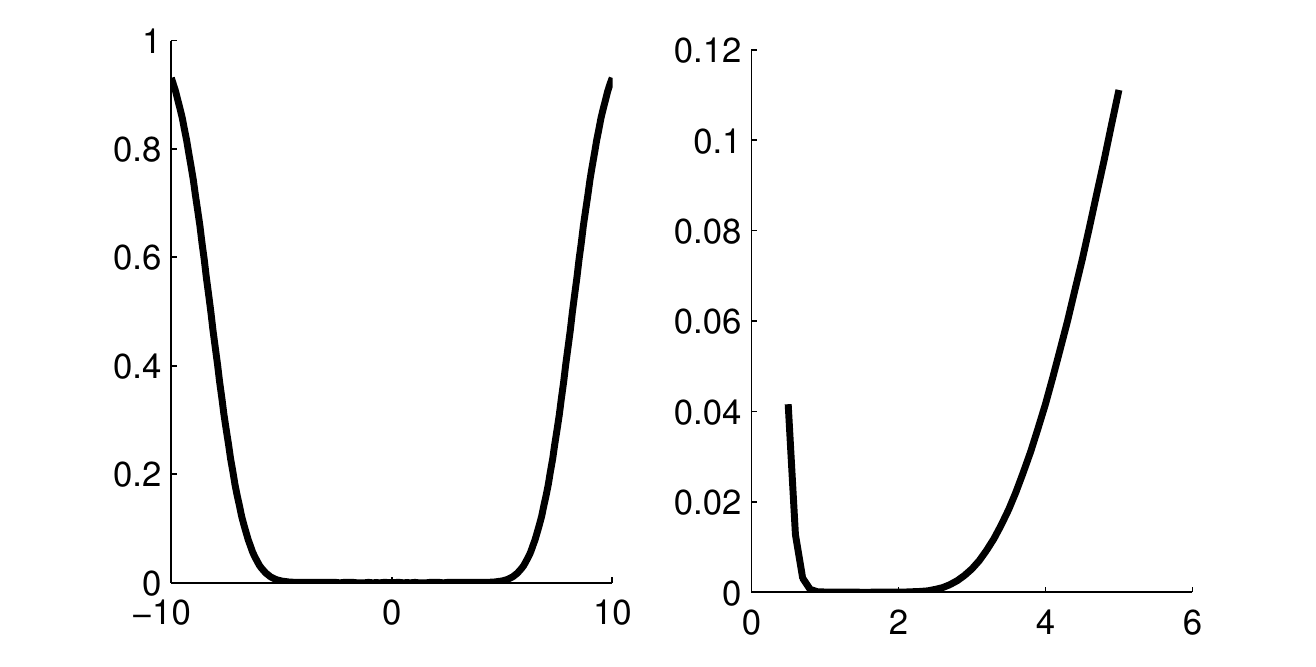}
		\put(10,50){$d$}
		\put(54,50){$d$}
		\put(49,3){$\mu$}
		\put(93,2.5){$\sigma$}
	\end{overpic}
\parbox{14cm}{\caption{\label{fig:error-v} \footnotesize
Comparison of the density $p$ of a normal distribution with mean $\mu$ and variance $\sigma^2$
with its approximation $\hat{p}=\sum_{i=1}^{n} ( p, e_i ) e_i$ for different
choices of $\mu$ and $\sigma$ with $n=20$: the graphs show the distance
$d:=( \int (\hat{p}-p)^2 dx)^{1/2}$ as function of $\mu$ and as function of $\sigma$
with fixed $\sigma=\sqrt{2}$ (left) and $\mu=0$ (right). The approximation is bad if
$\mu \not \in (-5,5)$  (left) or $\sigma \not \in (0.9,2)$ (right). The adaptive Galerkin
method overcomes this difficulty.}}
 \end{center}
\end{figure}

Assume for simplicity that $l=1$ and that the basis $\{e_i\}\subset D(\mathcal{A}^*) $ of
$H$ consists of orthonormal functions. We consider the equidistant time discretisation
given by $t_k=kT/K$, $k=0,\dots,K$. The standard Galerkin approximation computes
$\Upsilon_k$ at each time $t_k$. The AGA additionally adapts the location  $\mu_k$ and
the scale $\sigma_k>0$ of the basis by choosing appropriate values for these parameters
at every time step. Hence  the method works with the adapted basis $\{e_1^k, e_2^k,
\dots\}$ given by
\begin{align}\label{eq:adapted-basis}
{e}_i^k(x) :=\frac{1}{\sqrt{\sigma_k}}e_i\Big(\frac{x-\mu_k}{\sigma_k}\Big), \quad x\in \R.
\end{align}
Similar to \eqref{eq:GC-a}, we denote by $A^k$, $B^k$, $C^k$ and $D^k$ the matrices given
by
\begin{align}\label{eq:ABCG}
 a_{ji}^k=( \e^k_i,\mathcal{A}\e^k_j),\quad
 b_{ji}^k=( \e^k_i, h\e^k_j) ,\quad
 c_{ji}^k=(\e^k_i, (\lambda-1)\e^k_j), \quad
 d_{ji}^k =(\e^k_i, \e^k_j).
\end{align}
In algorithmic form the adaptive Galerkin approximation can be described as follows:

\begin{algorithm}[AGA]\label{algo:AGA}

\textbf{1. Initialization}:
\begin{itemize}
\item[i)] Set $\mu_0$ and $\sigma_0$ using the initial density: $\mu_0=\int
    xp_0(x)dx$ and $\sigma_0 = \big(\int (x -\mu_0)^2 p_0(x)dx \big)^{1/2}$, and
    define the basis functions $e_i^0 $, $1 \le i \le n$, as in
    \eqref{eq:adapted-basis}.
\item[ii)]  Compute $A^0$, $B^0$, $C^0$ and $D^0$ according to \eqref{eq:ABCG}.
\item[iii)] Compute $\U_0=(\psi_{0,1},\dots,\psi_{0,n})$ by
    \eqref{eq:approximation-qn}: $\psi_{0,i}=( p_0, {e}^0_i)$.
\end{itemize}

\pagebreak

\noindent\textbf{2. Iteration}: For $k=1,\dots , K$ do the following steps.
\begin{itemize}
\item[i)]  Compute $\tilde \U_k = (\tilde \psi_{k,1}, \dots, \tilde
    \psi_{k,n})^\prime $ from $\U_{k-1}$ applying Algorithm \ref{algo:EM} or
    \ref{algo:splitting-up}, using the basis functions $e_i^{k-1}$, $1 \le i \le n$.
\item [ii)] Compute the following estimates of the conditional mean and standard
    deviation:
$$
\hat{x}_{k}^{(n)}=\frac{\sum_{i=1}^{n}\tilde \psi_{k,i} (x,e_i^{k-1}) }
  {\sum_{i=1}^{n}\tilde \psi_{k,i}^{(n)}({1},e_i^{k-1})} \quad  \text{ and }  \quad
  \hat \sigma_{k}^{(n)} =
\bigg ( \frac{\sum_{i=1}^{n}\tilde \psi_{k,i} (x^2,e_i^{k-1}) }
  {\sum_{i=1}^{n}\tilde \psi_{k,i} ({1},e_i^{k-1})} - (\hat{x}_{k}^{(n)})^2 \bigg )^{\frac{1}{2}}.
$$

If $|\hat x^{(n)}_k - \mu_{k-1}|$ and $|\hat \sigma^{(n)}_k- \sigma_{k-1}|$ are
    smaller than a given threshold, set $\U_k = \tilde \U_k$,  $\mu_k = \mu_{k-1}$,
    $\sigma_k = \sigma_{k-1}$, $e_i^k := e_i^{k-1}$, $1 \le i \le n$. Let $k=k+1$ and
    continue with the iteration (Step~2).
\item [iii)] Otherwise do a \emph{transition of the basis} as follows: let $\mu_k :=
    \hat x^{(n)}_k$, $ \sigma_k := \hat \sigma_{k}^{(n)} $, define the new basis
    functions  as in \eqref{eq:adapted-basis} and compute the matrices $A^k$, $B^k$,
    $C^k$ and $D^k$ according to \eqref{eq:ABCG}. Finally, compute $\U_k$ by
    projecting $\tilde q_{t_k}  = \sum_{i=1}^n \tilde \psi_{k,i} e_i^{k-1}$ on the
    new basis: let $\psi_{k,i} = ( \tilde q_{t_k},\, e_i^k )$ and set $\U_k =
    (\psi_{k,1}, \dots , \psi_{k,n})^\prime$. Let $k=k+1$  and continue with
  Step~2.
\end{itemize}
\end{algorithm}

The AGA provides better results compared to the standard Galerkin approximation (see the
numerical experiments in Section~\ref{Simulation Results} below) while it is typically
more time consuming since the coefficient matrices in  \eqref{eq:ABCG} need to be
recomputed at every transition of the basis. However, in the AGA with respect to the
Hermite basis the corresponding terms can be computed explicitly which leads to an
efficient implementation of the AGA. In particular,
when the  coefficients $b$, $\sigma^2$, $h$ and $\lambda$ are
of polynomial type,  the corresponding coefficients can be computed explicitly with the
aid of Lemma \ref{lem:4.6}.

\subsection{The multi-dimensional case}\label{sec:multi}

In this section we shortly sketch the extension to the multi-dimensional case. For an introduction of multi-dimensional Hermite
polynomials, we refer to \citeN{bib:Berkowitz-Garner}.
Here, we proceed by the following method: let $\{e_1,e_2,\dots\}$
denote the Hermite bases defined in
Equation \eqref{eq:def-hermite-basis}. This constitutes a  basis of $L^2(\R)$. Hence,
\begin{align*}
\Big\{e_{i_1}\otimes e_{i_2} \otimes\cdots \otimes e_{i_d}: \,
i_1,i_2,\ldots i_d \in \N \Big\}
\end{align*}
is a Hilbert basis of $L^2(\R^d)$, where the tensor product is defined by $(e \otimes f)(x_1,x_2) :=e(x_1) f(x_2)$.
We choose the following  $m=n^d$-dimensional subspace {$H_m$}
\begin{align*}
H_m=\text{span}  \Big\{e_{i_1}\otimes e_{i_2} \otimes\cdots \otimes
e_{i_d}: \, i_1,i_2,\ldots i_d \in \{1,\ldots,n\} \Big\}.
\end{align*}
{From now on the (adaptive) Galerkin approximation works similarly as in the one-dimensional case, compare for instance  Algorithm \ref{algo:AGA}.}

\section{Numerical  Experiments}\label{Simulation Results}

In this section we present results from a number of numerical case studies. The aim is
to assess the performance of the Galerkin approximation relative to other methods (mostly
particle filters)  and to illustrate various practical aspects of the method such as the pros and cons of different basis functions and discretization schemes.

\paragraph{General description.} The basic
setup of each numerical experiment is as follows. In Step~1 a trajectory $x = (x_{t})_{0
\le t \le T}$ of the signal process \eqref{eq:state}  was generated using the
Euler-Maryuama method. In Step~2 we generated  for the given trajectory $x$ from Step~1 a
trajectory $z$ of the continuous observation \eqref{eq:Z} and a trajectory $n$ of the
point process observation $N$. In Step~3 various variants of the Galerkin approximation
were used to solve the corresponding Zakai equation for the conditional filter density.
For comparison purposes the filter problem was also solved using a particle filter.

The performance of the  numerical filtering algorithms was  assessed in different ways:
\begin{itemize}
\item By design, the mean $\hat x_t$ of the filter distribution at time $t$ minimizes
    the ${L}^2$-distance between the unobserved state $X_t(\omega)$ and the space
    $L^2(\Omega,\cF_t^{Z,N},\P)$. This suggests the following performance criterion:
    Generate $m$ independent trajectories $x^j, z^j, n^j$, $1 \le j \le m$ and solve
    numerically the ensuing filter problem for discrete time points
    $t_1,\ldots,t_K$. Compute  the so-called  root mean square error
    (abbreviated RMSE) given by
     \begin{align*}
\operatorname{RMSE}=\Big ( \frac{1}{mK}\sum_{j=1}^{m} \sum_{k=1}^{K} \|
X^j(t_k)-\hat{x}^j(t_k)\|^2\Big )^{\frac{1}{2}},
\end{align*}
Obviously, a filtering method that leads to a smaller RMSE can be considered to be
more accurate.
\item We can plot individual trajectories $x$,  $\hat{x}$ and $\hat \sigma^2$ of the signal, and of the  conditional mean and  variance of the filter distribution. This
    allows for a pathwise comparison of different numerical methods.
\item Finally, in some cases (e.g. Kalman-filtering for linear Gaussian models) the
    filter density is known explicitly. In those cases we can compare the
    filter density $p_{t}(\cdot)$ and the approximation $p^{(n)}_{t}(\cdot)$ that is
    obtained by normalizing the  numerical solution of the Zakai equation (see
    \eqref{eq:aga-density}).
\end{itemize}

Our numerical experiments with a one-dimensional signal process  use the setup of
Example~\ref{exam:linear_model}. In this case   $X$ is a one-dimensional
Ornstein-Uhlenbeck process with mean-reversion parameter $b$ and volatility $\sigma$ and $N$ is a doubly stochastic Poisson process with intensity $\lambda(X_t)$. The functions
$h(\cdot)$ and $\lambda(\cdot)$ are of the form  $h(x) =  h x$ and $\lambda(x) =
 \lambda x^2$. The parameter
values are  as follows: $b=0.5$, $\sigma=2$, and the initial distribution of $X$ is
normal with $\mu=5$, $\sigma^2 =0.01$. Unless stated otherwise we took a time step
$\Delta =10^{-6}$ (essentially continuous observations). The values of $h$ and
$ \lambda$ vary with the experiments and are hence given in the captions of the
graphs and tables.

We also considered the case of a multidimensional signal process of dimension $d=5$. We
assumed that the signal process has dynamics $ dX_t = b X_t dt + \sigma d V_t$ for a
$3$-dimensional Brownian motion $V$. The observation process is three-dimensional and
$h(x) = \tilde h x$. The matrices $b$, $\sigma$ and $\tilde h$ are as follows:
\begin{align*} {\tiny
b=\left(
  \begin{array}{rrrrr}
    1   &  0 &     0  &   1 &    0\\
     1   &  1 &   -1   &  0 &    1\\
     0    & 1  &  -1   & -1 &   -1\\
     0    &-1   & -1   &  1 &    1\\
     1    &-1    & 0   &  0 &    1\\
  \end{array}
\right), \,\,\sigma=\left(
  \begin{array}{ccccc}
    1 &    0   &  1 \\
    2 &    1  &   1 \\
    1 &    1  &   1 \\
    1 &    1  &   1 \\
    0 &    0 &    1 \\
  \end{array}
\right),\,\, \tilde h=\left(
  \begin{array}{ccccc}
    0.2  &  0.3  &  0.2  &  0.3  &  0.4\\
    0.2  &  0.1  &  0.2  &  0.1  &  0.2\\
    0.2  &  0.2  &  0.4  &  0.2  &  0.2\\
  \end{array}
\right), }
\end{align*}
and   $N$ is a one-dimensional doubly stochastic Poisson process with intensity $0.1
(X^1_t)^2+0.2 (X^2_t)^2+0.3(X^3_t)^2+0.1 (X^4_t)^2 +0.1 (X^5_t)^2$. The
basis functions were chosen as described in Section \ref{sec:multi}.

\paragraph{Results.} In the following we summarize the key findings from our numerical
experiments; the outcome of each experiment is described in detail in the captions of Figures  \ref{fig:estimation-linear-strong2}  -- Figure \ref{fig:estimation-multi} below.
\begin{enumerate}
\item For a one-dimensional state variable process the adaptive Galerkin
    approximation performs very well: given a sufficient number of basis functions
    the precision is equal to the precision of a particle filter, but the computation
    time is significantly lower. This can be seen from inspection of
    Table~\ref{tab:case-linear-1}, where we give the RMSE and the computation time
    for various filtering algorithms and parameter values. The performance of the
    Galerkin approximation and the additional value for the estimation of $X$ obtained by  incorporating the observation of the point process $N$ is further illustrated in
    Figure~\ref{fig:estimation-linear-strong2}. In
    Figure~\ref{fig:estimation-linear-KBt1} we consider  the special case of the
    Kalman-Bucy filter (no point process observations). Here the filter density is
    known explicitly and we can compare the approximation obtained via Galerkin
    approximation to the correct density. The figure clearly shows that the Galerkin
    approximation provides a good approximation to the overall density (and not just
    to the conditional mean $\hat x_t$).
\item The adaptive Galerkin method can bring a substantial performance
    enhancement if we consider examples with small observation noise and hence with
    rapidly moving scale and location of the filter distribution. This is 
    illustrated in Figure~\ref{fig:estimation-linear-nb}.
\item While computationally more involved, the splitting-up approximation is
    significantly more accurate and stable if the time-discretization step $\Delta$ is moderately
    large, as is clearly shown in Figure~\ref{fig:estimation-SU-vs-Euler}. At this
    point we would like to stress that in many applications of filtering, observations
    arrive at discrete time points such as daily observations, and that one resorts
    to continuous-time filtering methods merely for convenience. This implies that
    $\Delta$ cannot be freely chosen by the analyst and it is important to have
    numerical methods that are robust with respect to  the choice of $\Delta$.
\item Figure~\ref{fig:estimation-multi} for the case $d=5$ indicates that the
    Galerkin approximation works reasonably well also for a higher dimensional signal
    process. However, the number of basis functions increases exponentially in $d$
    (at least for the basis chosen as in Section \ref{sec:multi}). It would be
    interesting to see if a further performance enhancement is possible if we choose
    a different basis, but this is left for further research.
\end{enumerate}

\paragraph{Comparison of the Gaussian and the Hermite basis functions.} Finally we briefly discuss the pros and cons of Gaussian and  Hermite basis functions.

The main advantage of the Hermite basis is clearly the fact that the $e_i$ form an orthonormal system. This facilitates the numerical implementation of the method, as there is no need to invert the matrix $D$ introduced in  \eqref{eq:GC-a}. Moreover, the orthogonality of the $e_i$ implies that the subspace spanned by the first $n$ Hermite basis function is in some sense `larger' than the subspace spanned by the first $n$ Gaussian basis functions, so that it is possible to obtain a similar level of accuracy with a smaller size of the basis. Table~\ref{tab:simulations} shows that, at least for the case of  the Kalman filter with point process observations, these advantages are not only theoretical: in order to reach a similar level of accuracy the methods based on Hermite basis functions require a  substantially lower computational time than the algorithm with Gaussian basis functions. A second advantage of Hermite polynomials is the fact that for certain models it is possible to compute the coefficient matrices for the Galerkin approximation analytically, as was illustrated in Example~\ref{exam:linear_model}. This is particularly useful if one wants to use some form of the adaptive Galerkin approximation, since in that case the coefficient matrices need to be recomputed during the filtering procedure.

As mentioned in \citeN{bib:ahmed-radaideh-97}, the use of Hermite basis functions might in principle lead to negative values for the conditional density which is clearly undesirable. To this we mention that we never encountered the problem in our numerical experiments unless we took the number of basis functions extremely small. But  we readily admit that in order to exclude the possibility of negative filter densities a priori,  one has to resort to a set of nonnegative basis functions such as the Gaussian basis proposed by \citeN{bib:ahmed-radaideh-97}, even if this might  lead to a higher computational effort.

\begin{table}
\begin{center}
\begin{tabular}[h]{llll}
\toprule
 $N_G$/$N_P$ & 5/20 & 10/50 &  15/100   \\\midrule
AGAH(EM)  &0.63  (0.1s) &   0.42 (0.1s) &  0.42  (0.1s) \\
AGAH(SU)  & 0.65 (2.4s) &   0.43 (3.1s) &  0.43  (3.9s) \\
 PF      & 0.46 (9s)   &   0.46  (22s) &  0.42  (46s)  \\
  \bottomrule
\end{tabular}

\vspace{4mm}

\parbox{14.8cm}{
\caption{\label{tab:case-linear-1} Performance comparison for different filter
algorithms: we plot the RMSE  and in brackets the computation time for two  Galerkin
filters and a particle filter.  Here $N_G$ represents the number of basis functions in
the Galerkin approximation and $N_P$ the number of particle in the particle filter.
AGAH(EM) respectively AGAH(SU) stands for the adaptive Galerkin approximation  with respect to the Hermite basis, where we used the Euler-Maruyama (EM) approximation and the  splitting-up approximation (SU) for Equation \eqref{eq:Upsilon-e}, respectively.
The parameter  values $h= \lambda =0.1$ used in the experiment correspond to a relatively uninformative observation filtration; in computing the RMSE we used $m=100$.
}}
\end{center}
\end{table}

\begin{figure}[t]
\begin{center}
\vspace{-1cm}
\parbox{18cm}
{
	\begin{overpic}[width=8cm,height=5.5cm]{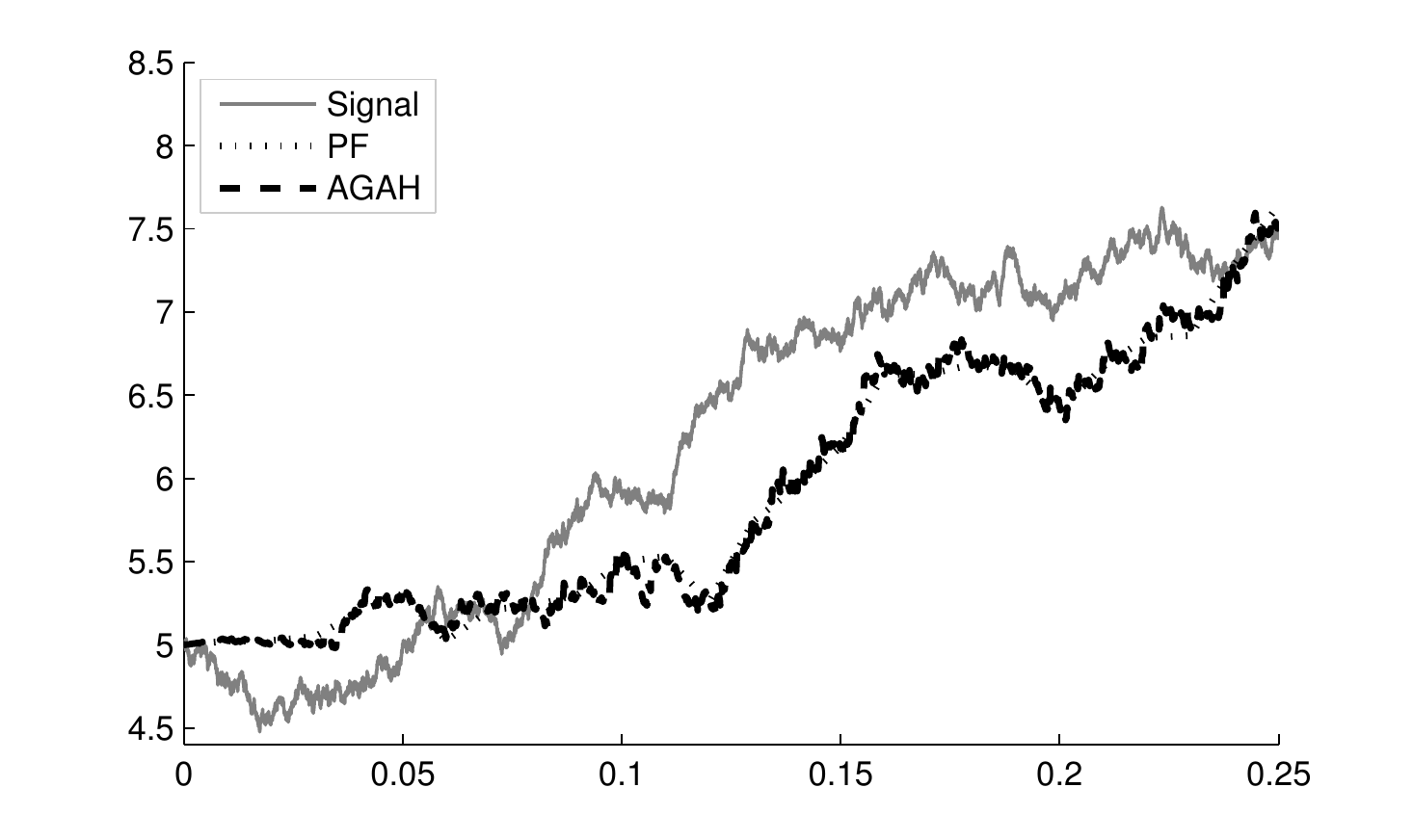}
		\put(8,67){$\hat x_t$}
		\put(94,3){$t$}
	\end{overpic}
	\begin{overpic}[width=8cm,height=5.5cm]{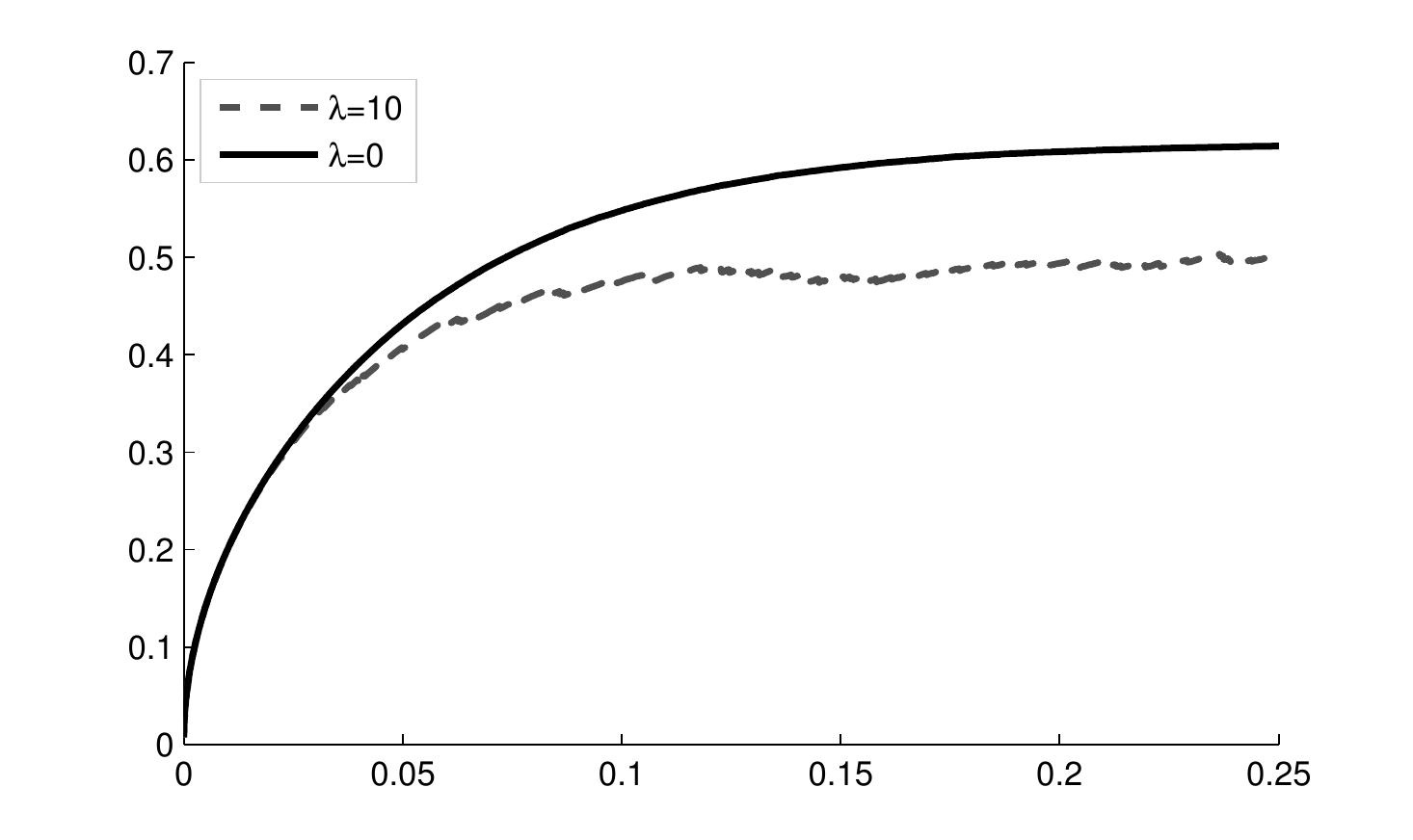}
		\put(8,67){$\hat \sigma_t$}
		\put(94,3){$t$}
	\end{overpic}
 }
 \parbox{15.8cm}{
\caption{\label{fig:estimation-linear-strong2}\emph{Illustration of filtering  and the value of point process information}. We choose  $ h=5.5$,    $ \lambda=10$.
\emph{Left:} trajectories of $X$ and of $\hat x$ using the Galerkin method and particle filtering. Both
methods perform well.
In the \emph{right} graph we illustrate the gain of using \emph{point process observations}:  we plot the trajectory of the conditional standard deviation $\hat \sigma_t$ for the case with only continuous observation $\lambda =0$ and with continuous and point process observations ($\lambda =10$, lower trajectory).
The approximation by the two methods are  very close in that case.
Clearly, including point process information reduces  the conditional
standard deviation  significantly in this example.  }}
\end{center}
\end{figure}

\begin{figure}[p.21]
\begin{center}
\begin{minipage}{18cm}\hspace{-1cm}
	\begin{overpic}[width=18cm,height=5.5cm]{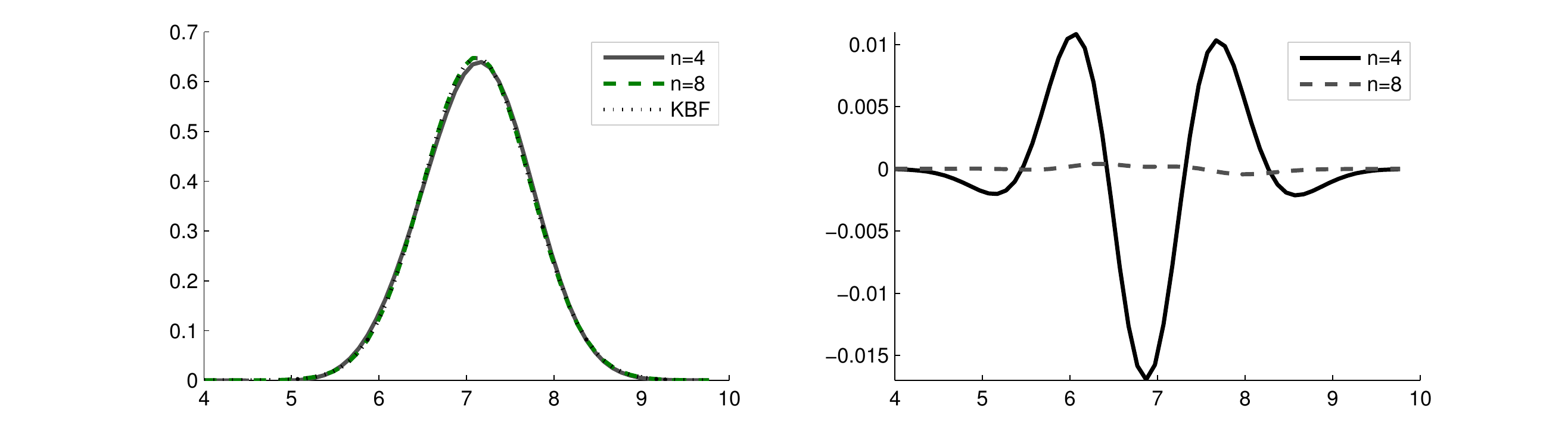}
		\put(91.5,1){$x$}
		\put(24,32){Filter densities}
		\put(68,32){Differences}
		\put(47.5,1){$x$}
	\end{overpic}
\end{minipage}
 \parbox{15.8cm}{
\caption{\label{fig:estimation-linear-KBt1}
\emph{Comparison of the  theoretical filter density  and the  AGAH.} We consider an $n$- dimensional  Hermite basis and the case of purely continuous
observations ($\lambda =0$). In this case, the filter problem has an explicit solution that can be computed with
the  Kalman-Bucy filter (KBF). \emph{Left:} filter densities; \emph{right:} approximation error.
The adaptive Galerkin approximation (AGAH)  is very close to the explicit solution for $n \ge 8$.   }}
\end{center}
\end{figure}

\begin{figure}\vspace{-1cm}
\begin{center}
\parbox{18cm}{
	\begin{overpic}[width=8cm,height=5.5cm]{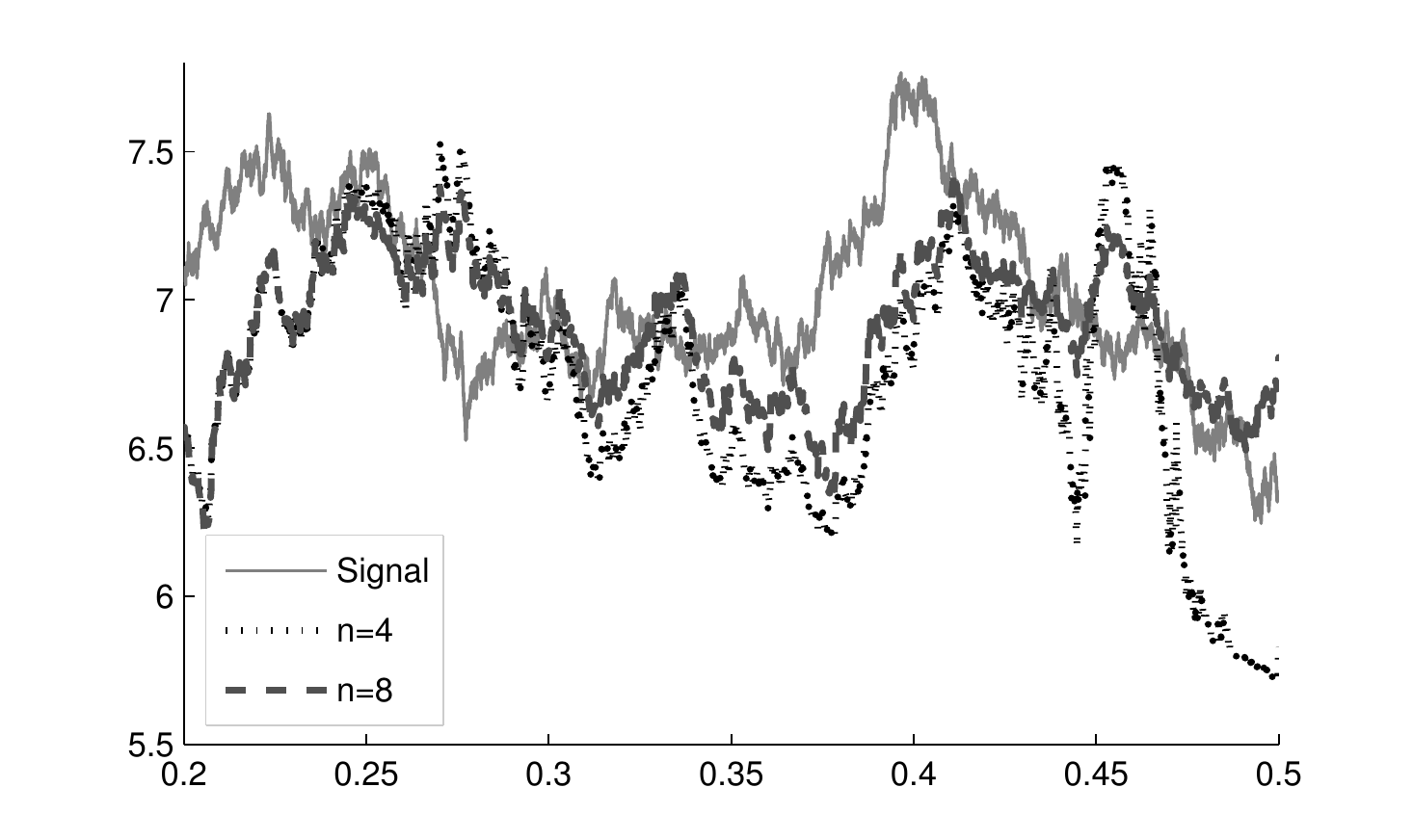}
		\put(8,67){$\hat x_t$}
		\put(94,3){$t$}
	\end{overpic}
	\begin{overpic}[width=8cm,height=5.5cm]{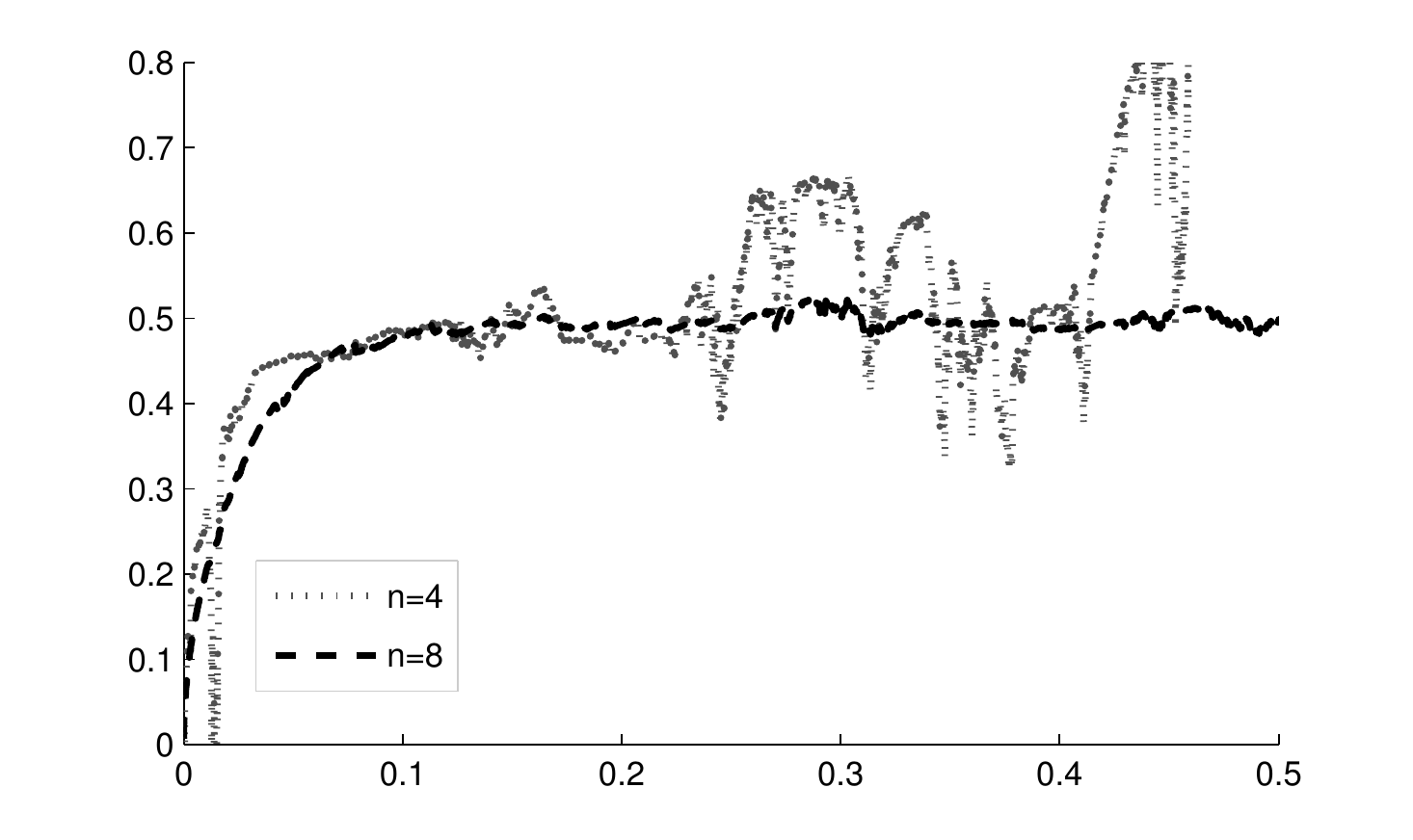}
		\put(8,67){$\hat \sigma_t$}
		\put(94,3){$t$}
	\end{overpic}
}
 \parbox{15.8cm}{\caption{\label{fig:estimation-linear-n} Filter estimate (left) and conditional standard deviation (right)
for a varying number $n$ of Hermite basis functions in the adaptive Galerkin approximation (AGAH).
Here $h=5.5$, $\lambda=10$ and we use the  AGAH  with $n=4$ and $8$  basis
functions.  The case $n=4$ shows a bad performance; the filter with $8$ basis functions performs reasonably well.
The right plot indicates that plots of the conditional variance
can be a useful tool for determining if the number of basis functions used is  appropriate.}}
\end{center}
\end{figure}

\begin{figure}
\begin{center}
 \parbox{18cm}{
	\begin{overpic}[width=8cm,height=5.5cm]{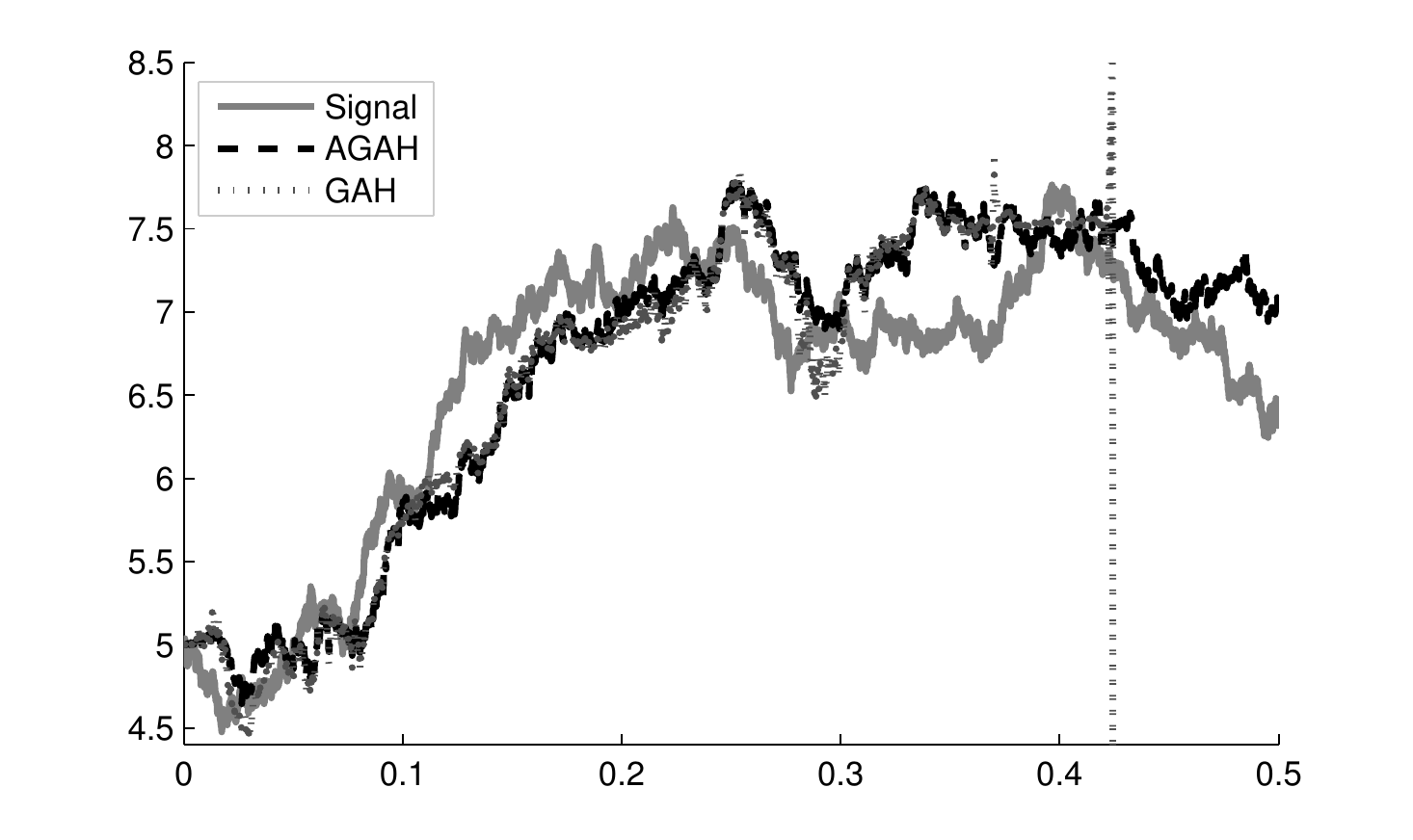}
		\put(8,67){$\hat x_t$}
		\put(94,3){$t$}
	\end{overpic}
	\begin{overpic}[width=8cm,height=5.5cm]{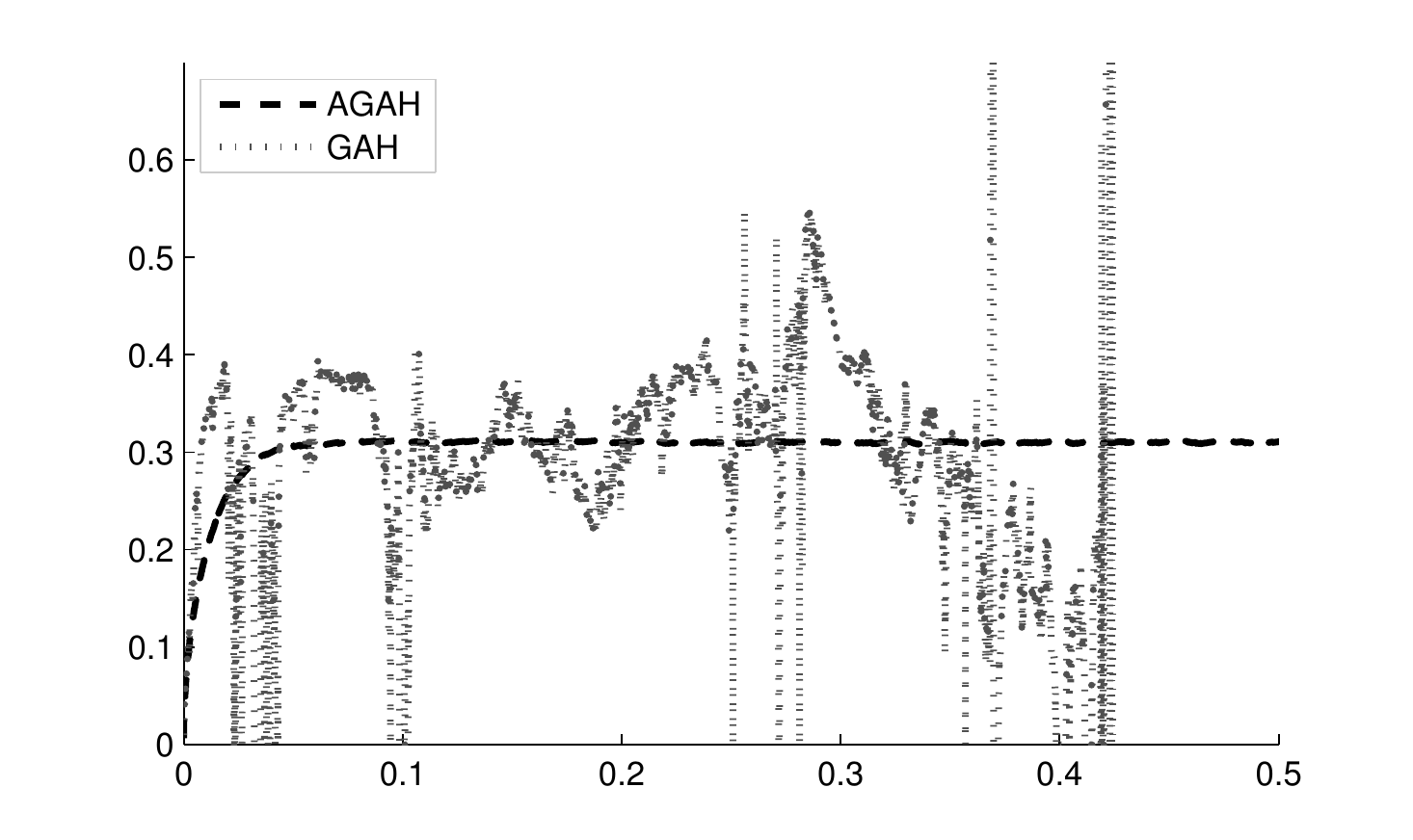}
		\put(8,67){$\hat \sigma_t$}
		\put(94,3){$t$}
	\end{overpic}
}
 \parbox{15.8cm}{\caption{\label{fig:estimation-linear-nb}  \emph{Comparison of the ordinary Galerkin approximation (GAH)
 with  the adaptive Galerkin approximation (AGAH)}. Both approximations work with $n=20$  Hermite basis functions.
In this example  $h=20$, $\lambda=10$ (these parameter values correspond to a very low observation noise).
Note that  the ordinary Galerkin filter  performs poorly, whereas the AGAH  performs reasonably well. }}
  \end{center}
\end{figure}


\begin{figure}[p.22]
\begin{center}
\begin{tabular}{cc}
	\begin{overpic}[width=8cm,height=5.5cm]{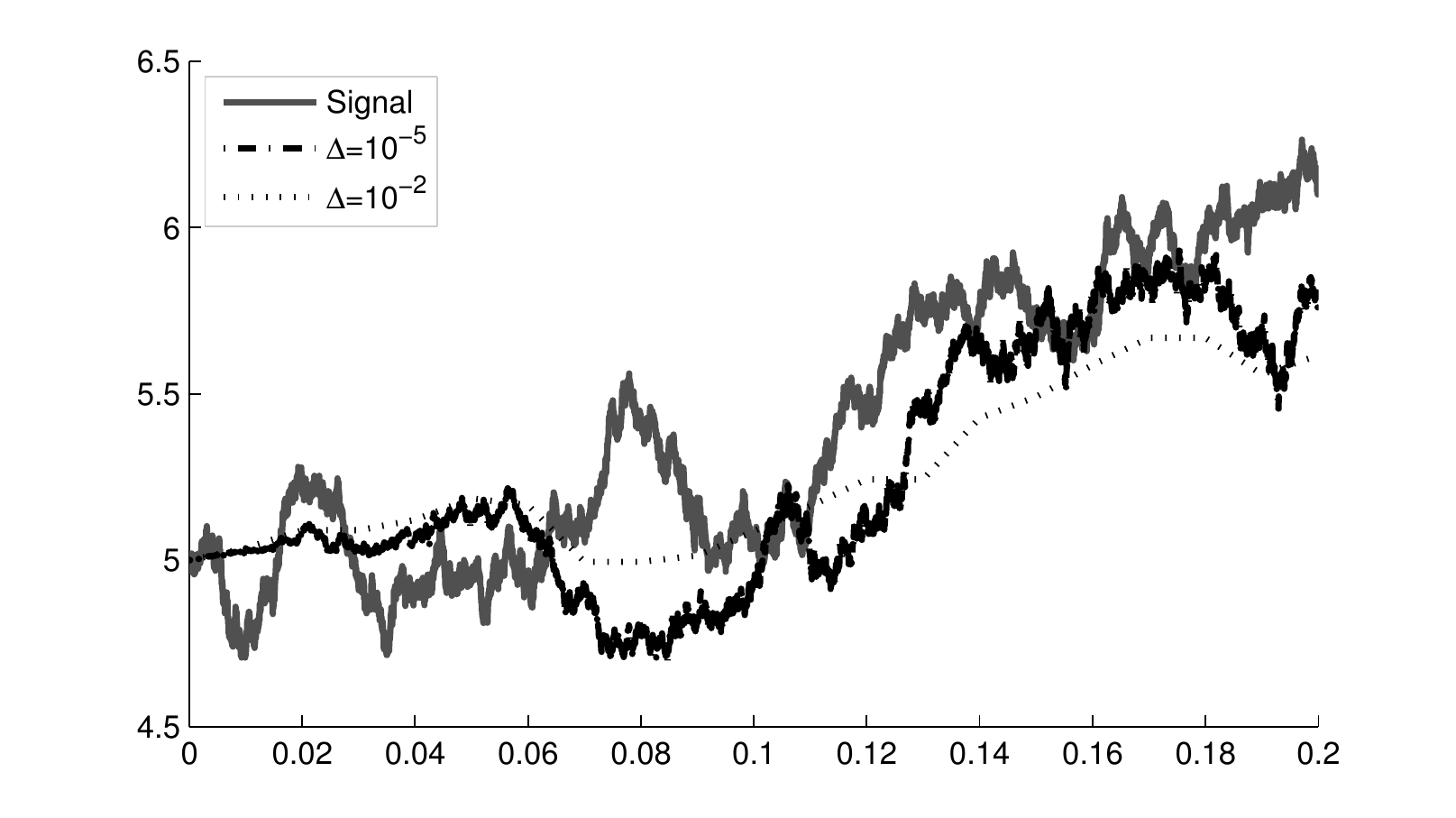}
		\put(24,75){Splitting-up approximation}
		\put(8,67){$\hat x_t$}
		\put(94,3){$t$}
	\end{overpic}
	\begin{overpic}[width=8cm,height=5.5cm]{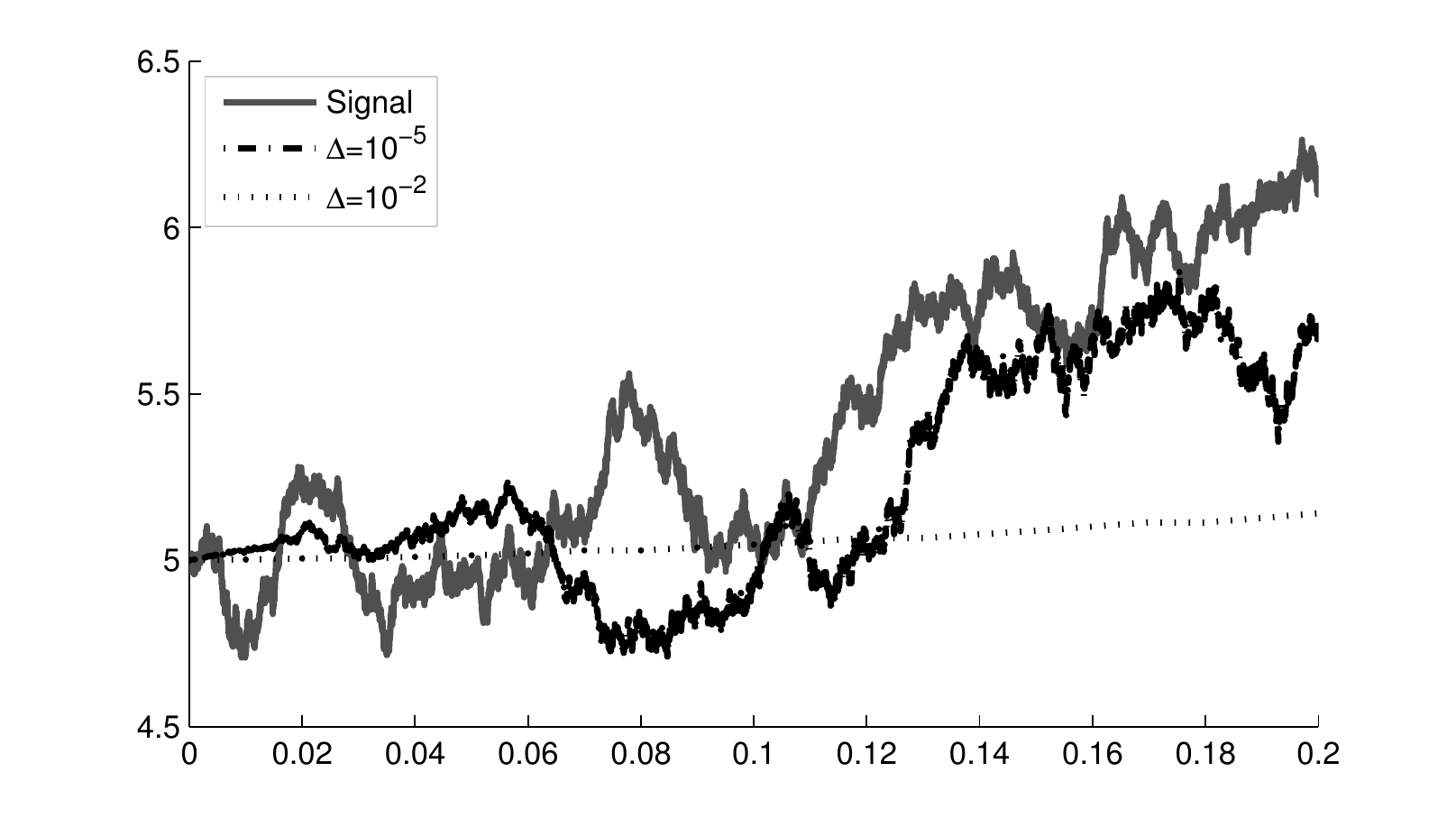}
		\put(20,75){Euler-Maruyama approximation}
		\put(8,67){$\hat x_t$}
		\put(94,3){$t$}
	\end{overpic}
 \end{tabular}
\parbox{15.8cm}{
\caption{\label{fig:estimation-SU-vs-Euler} \emph{Comparison of
splitting-up approximation (left) and Euler-Maruyama approximation (right).}  In
this figure, we compare the results obtained by these two methods
for different  $\Delta$.  The results obtained coincide for $\Delta=10^{-5}$,  but the splitting-up approximation is more accurate  when $\Delta$ is large. In particular, the splitting up method
provides a good estimate even for $\Delta=10^{-2}$.   }}
\end{center}
\end{figure}

\begin{figure}
\begin{center}
\begin{tabular}{cc}
	\begin{overpic}[width=6cm,height=7.5cm]{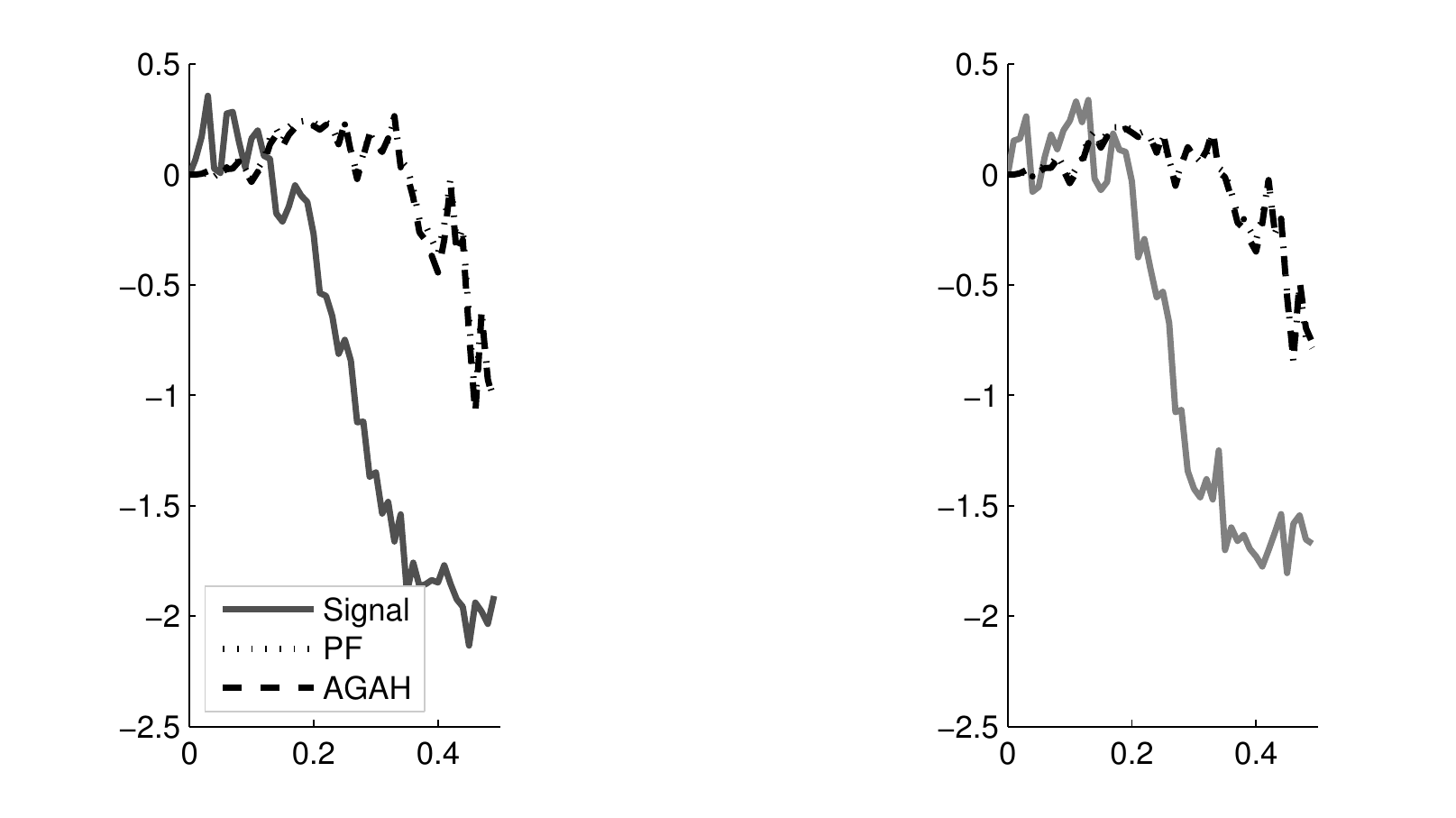}
		\put(21,97){$\hat x_t$}
	\end{overpic}
	\begin{overpic}[width=5.2cm,height=7.5cm]{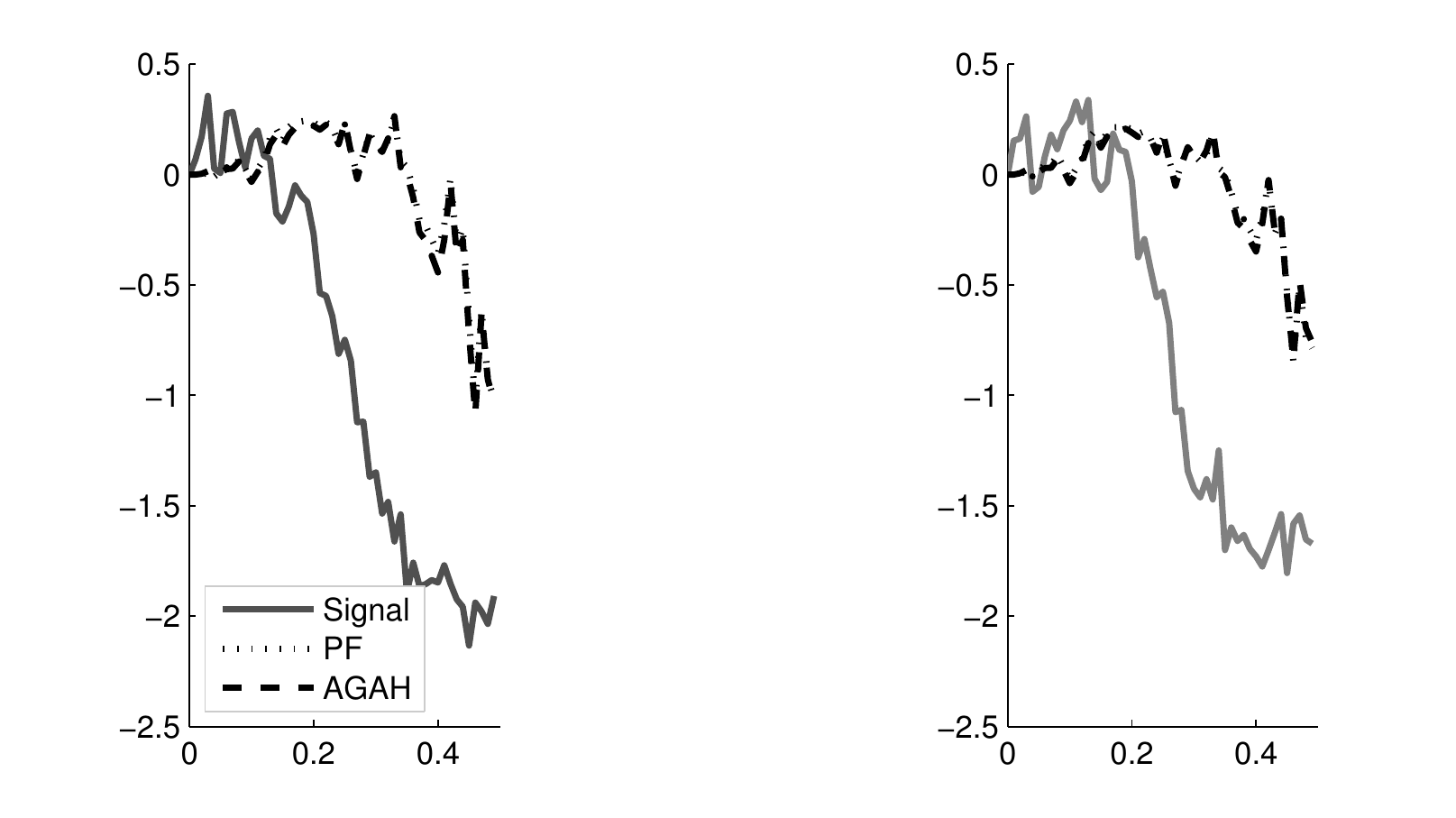}
	\end{overpic}
	\begin{overpic}[width=5cm,height=7.5cm]{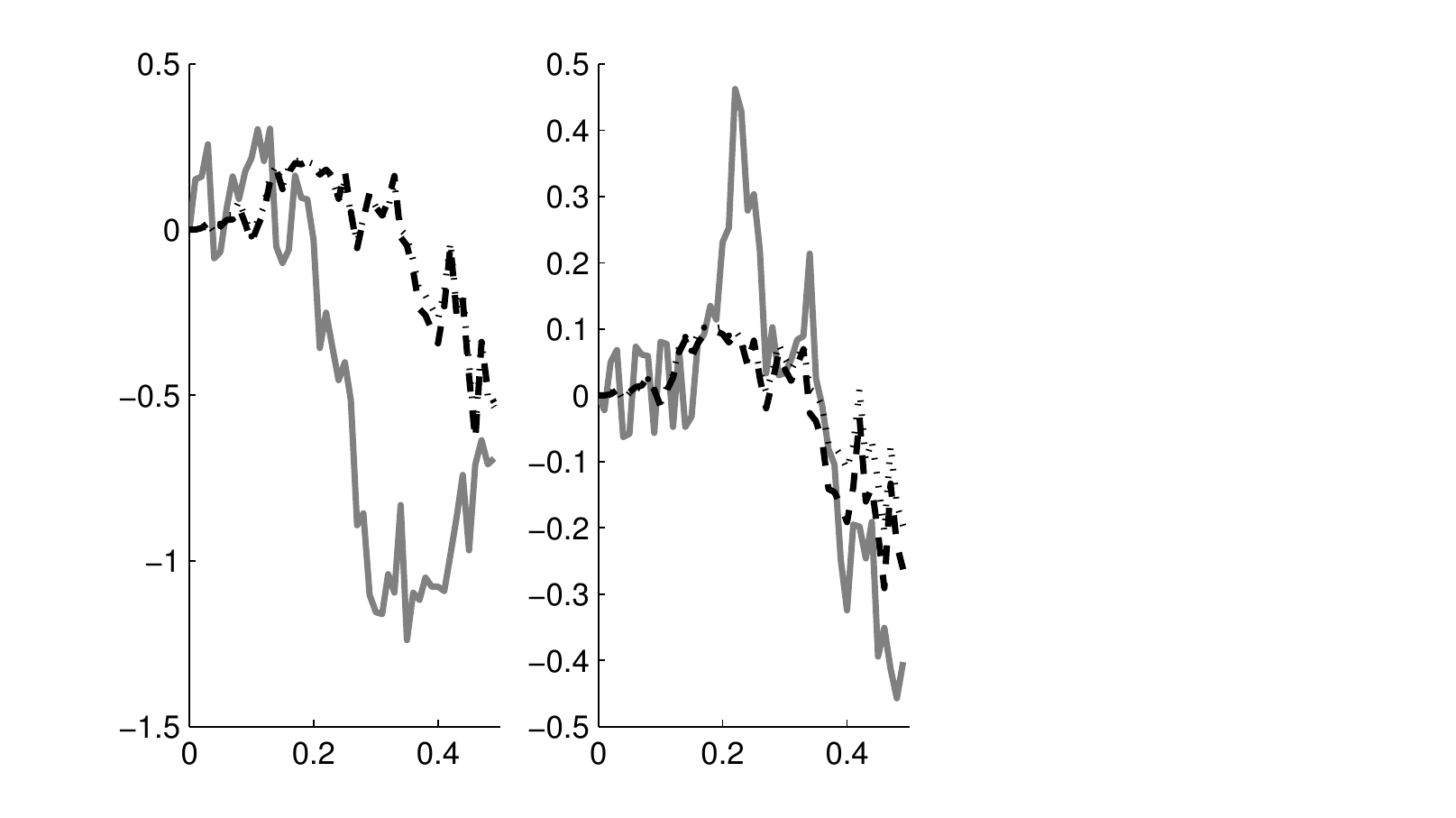}
		\put(56,5){$t$}
	\end{overpic}
\end{tabular}
 \parbox{15.8cm}{\caption{\label{fig:estimation-multi}  \emph{The multidimensional case.} Comparison of the adaptive Galerkin approximation (AGAH)
with the particle  filter (PF) for the case of a multi-dimensional signal process $X$. We show plots of the conditional mean for
a basis of size $n=4^5=1024$.
For the  particle filter we took $10^3$ particles.  The computation time was 12 seconds (PF) and 9
seconds  (AGAH). The results obtained by both methods are very close to each other. }}
\end{center}
\end{figure}

\clearpage

\begin{table}[h]
\begin{center}
\begin{tabular}{l|ll|ll|ll|lll} \toprule
 & $n$ & Time & RMSE & s.e. & EDM & s.e. & EDV & s.e. & \\
 \midrule
PF  &100 & 41.8 & 0.1685 & 0.4774 & 0.0179 & 0.0621 & 0.003 & 0.0122 &  \\
     &500 & 183 & 0.1518 & 0.4285 & 0.0039 & 0.0166 & 8e-04 & 0.0037 &  \\
     &1000 & 366 & 0.1530 & 0.4325 & 0.0024 & 0.0109 & 4e-04 & 0.0016 &  \\
\midrule
GAG  
     &25 & 10.9 & 0.3869 & 1.2997 & 0.2316 & 1.0390 & 0.6159 & 2.5379 &  \\
     &30 & 16.7 & 0.1632 & 0.4445 & 0.0151 & 0.1061 & 0.0030 & 0.0959 &  \\
     &40 & 27.1 & 0.1641 & 0.4495 & 0.0159 & 0.1235 & 0.0044 & 0.1385 &  \\
     &50 & 46.8 & 0.1850 & 0.6182 & 0.0400 & 0.4320 & 0.2659 & 4.5433 &  \\
     &120 & 927 & 0.1580 & 0.4276 & 0.0075 & 0.0221 & 5e-04 & 0.0042 &  \\
\midrule
GAH  &8 & 1.9 & 9.1625 & 228.80 & 9.1831 & 231.83 & 0.0438 & 1.2288 &  \\
     &10 & 2.0 & 2.0444 & 42.221 & 1.8825 & 42.117 & 0.1273 & 6.3878 &  \\
     &12 & 2.1 & 0.6722 & 3.4018 & 0.5415 & 3.1160 & 0.0045 & 0.0932 &  \\
     &24 & 3.7 & 0.1494 & 0.4163 & 0.0005 & 0.0019 & 0.0001 & 0.0017 &  \\
     &40 & 7.2 & 0.1494 & 0.4162 & 0.0004 & 0.0015 & 8.7e-05 & 0.0004 &  \\
\midrule
AGAH  &8 & 23.3 & 0.1518 & 0.4232 & 0.0007 & 0.0029 & 0.0002 & 0.0009 &  \\
     &12 & 32.7 & 0.1494 & 0.4193 & 0.0006 & 0.0022 & 9.3e-05 & 0.0004 &  \\
     &16 & 50.8 & 0.1493 & 0.4193 & 0.0006 & 0.0023 & 9.8e-05 & 0.0004 &  \\
\bottomrule
\end{tabular}

\vspace{4mm}

\parbox{14.8cm}{
\caption[]{\label{tab:simulations}
\emph{Performance of the different algorithms in terms of root mean square error (RMSE).} We consider the particle filter (PF), the Galerkin approximation with Gaussian basis (GAG), the Galerkin approximation with Hermite basis (GAH) and the adaptive Galerkin approximation with Hermite basis (AGAH). The chosen parameter values are $h=5.5$, $\sigma=1$,and $\lambda=10$ with inital distribution   $\cN(2,1)$ (as in Figure \ref{fig:estimation-linear-nb}) and we consider the time-interval  $[0,0.5]$. Besides computation time and RMSE we plot the estimated deviation \footnotemark\ in mean (EDM) and estimated deviation in variance (EDV) by comparing the method to the solution of a particle filter with $10^4$ basis functions and their standard errors. A low value suggests that the method is very close to the exact solution. The methods with the Hermite basis (GAH, AGAH) outperform the methods with the Gaussian basis (GAG) by a large scale: they are at the same time faster and more accurate. In this case with average precision GAH with at least 24 basis functions provides  fast and accurate results. In a setup with lower observation noise, AGAH shows a better performance than GAH, compare Figure \ref{fig:estimation-linear-nb}.}}
\end{center}
\end{table}

\footnotetext{More precisely, we consider
 	\begin{align*}
 		\text{EDM} &:= \frac{1}{mL}   \sum_{j=1}^m \sum_{i=1}^L ( \hat{X}^j(t_i)-\tilde{X}^j(t_i) )^2  \quad \text{ and }\quad
		\text{EDV} := \frac{1}{mL}     \sum_{j=1}^m \sum_{i=1}^L ( \hat{V}^j(t_i)-\tilde{V}^j(t_i) )^2,
    \end{align*}
where $\hat{X}^j(t_i)$ is the filtering estimate at time $t_i$ in the $j$-th simulation  and $\tilde{X}^j(t_i)$ is the result provided by a particle filter with $10^4$ particles, which is close to the explicit solution of the problem. Similarly, $\hat{V}^j(t_i)$ is the conditional variance obtained by  filtering  at time $t_i$ in the $j$-th simulation and $ \tilde{V}^j(t_i)$ is the conditional variance obtained by  branching particle filter with $10^4$ particles. The number of simulations is $m=100$.}
\clearpage

\clearpage

\appendix
\section{Additional Proofs}
\label{app:Proofs}

\begin{proof}[Proof of Lemma \ref{lem:Lbounds}.]
 Our aim is to show that for all $n \in \N$  it holds that
\begin{align}\label{eq:A1}
\|L^{n}\|^{\frac{1}{n}} \leq\frac{\sqrt{T}\bar{S}(\|B\|^2+\|C\|^2)^{\frac{1}{2}}}{(n!)^{\frac{1}{2n}}}.
\end{align}
The operator $L$ was defined in \eqref{eq:conti-L}. We rewrite $L\xi$ as
$$ (L\xi)(t) =: \int_0^t E(t,s) \xi_{s-} d M_{s} $$
with the $(l+1)$-dimensional martingale $M:=(Z^\top,Y)^\top$ and $E(t,s):=S_{t-s}(B^\top,C)^\top$.
Iterative application of $L$ gives that
\begin{align*}
(L^{n}\xi)(t) &=
\int_0^t
 E(t,t_1) \bigg(\int_0^{t_1} E (t_1,t_2)\Big(\ldots  \int_0^{t_{n-1}} E(t_{n-1},t_n) \xi_{t_n} dM_{t_n}\ldots \Big)dM_{t_2} \bigg)dM_{t_1}.
 \end{align*}
To compute   $|L^{n} \xi|_{T}=\sup_{t\in[0,T]}\big(\E^0\big( \|(L^n\xi)(t)\|_H^2 \big)\big)^{\nicefrac{1}{2}}$,
 note that the quadratic variation of $M$ is $<M>_t = I_{l+1} t$ where $I_{l+1}$ is the identity matrix on $\R^{l+1}$.
The It\^o-isometry therefore yields
\begin{align*}
\lefteqn{\E^0\big( \|(L^n\xi)(t)\|_H^2 \big) } \quad \\
 &= \E^0 \bigg( \int_0^t \Big\|E(t,t_1) \bigg( \int_0^{t_1} E (t_1,t_2)\Big(\ldots  \int_0^{t_{n-1}} E(t_{n-1},t_n) \xi_{t_n}
dM_{t_n}\ldots \Big)dM_{t_2} \bigg) \Big\|^2 dt_1 \bigg)\\
& \leq  \bar{S}^2 (\| B\|^2+\|C\|^2)   \int_0^T  \E^0 \bigg(\int_0^{t_1} \Big\| E (t_1,t_2)\Big(\ldots  \int_0^{t_{n-1}} E(t_{n-1},t_n) \xi_{t_n}
dM_{t_n}\ldots \Big)dM_{t_2}  \Big\|_H^2 \bigg)dt_1\\
&\leq   \bar{S}^{2n} (\| B\|^2+\|C\|^2)^{n} |\xi|_T^2\cdot   \int_0^T \int_0^{t_1} \ldots \int_0^{t_{n-1}}  \ d{t_n},\ldots, d {t_1} \\
&=  \bar{S}^{2n} (\| B\|^2+\|C\|^2)^{n} |\xi|_T^2\cdot \frac{T^n}{n!}
\end{align*}
and we obtain \eqref{eq:A1}.
\end{proof}

\begin{proof}[Proof of Proposition \ref{cor7.3}.]
As remarked after Theorem \ref{Thm:convergence_linear_part}, in
\eqref{eq:convenientcondition} the claim is proved  if we can show that $\cup_n V_n$ is
dense in $V = H^1 (\R^d)$. Here $V_n=\text{span}\{e_1,\ldots,e_{n}\}$ and $V=H^1(\R^d)$.
Let $C_0^{\infty}$ be the set of smooth functions with compact support.   Then, by
Proposition 1 and Theorem 4 in \citeN{bib:Bongioanni-Torrea}, 
 there exist for all $u\in C_0^{\infty}$   a sequence $u_n\in \cup_n V_n$ such that
$\|u-u_n\|_V\rightarrow 0$  as $n\rightarrow \infty.$ Since $C_0^{\infty}$ is dense in
$V$ the claim follows.
\end{proof}

Next we turn to the determination of the coefficient matrices from Example \ref{exam:linear_model}. Our starting point are  well-known recursive reslationships  for  Hermite polynomials defined in \eqref{eq:def-hermite}: it holds that  $f_i^\prime = x f_i - f_{i+1}$ and $(f_i)^\prime = i f_{i-1}$. This gives
$xf_i(x) = i f_{i-1}(x)+f_{i+1}(x)$. These relations can be used to derive a number of useful relationships for the Hermite basis functions:
\begin{lemma} \label{lemma:auxiliary-ralations-Hermite}
For the Hermite basis $(e_i)_{i \ge 1}$ defined in \eqref{eq:def-hermite-basis} it holds that
\begin{align*}
x e_i(x) &= \sqrt{i-1} \, e_{i-1}(x) + \sqrt{i} \, e_{i+1} (x)\\
x^2 e_i(x) &= \sqrt{(i-1)(i-2)} \, e_{i-2}(x) +(2i-1) e_i(x) + \sqrt{i(i+1)} \, e_{i+2}(x) \\
(e_i)^\prime &= \frac{1}{2}\big( \sqrt{i-1} \, e_{i-1} - \sqrt{i} \, e_{i+1}\big) \\
x\, (e_i(x)) ^\prime &=  \frac{1}{2}\big( \sqrt{(i-1)(i-2)}\, e_{i-2} -e_i - \sqrt{i(i+1)}\, e_{i+2}\big).\\
(e_i)^{\prime \prime} &= \frac{1}{4}\big( \sqrt{(i-1)(i-2)} \, e_{i-2} + (1-2i) \, e_i + \sqrt{i(i+1)} \, e_{i+2}\big).
\end{align*}
\end{lemma}
\begin{proof}
Regarding the first result, note that
\begin{align*}
x e_i (x) &= \sqrt{\frac{\phi (x)}{(i-1)!}} (x f_{i-1} (x) ) =  \sqrt{\frac{\phi(x)}{(i-1)!}} ((i-1) f_{i-2}(x)+f_{i}(x)  )  \\
& = \sqrt{i-1} e_{i-1} (x)+ \sqrt{i} e_{i+1} (x).
\end{align*}
Using this expression twice on $x^2 e_i(x) = x \, (x  e_i (x))$, we  obtain the second result. For the remaining two expressions
we compute the derivative of $e_i$ and use the recursive expression for $f_i^\prime$ given above.  We obtain
\begin{align*}
(e_i(x)) ^\prime &= \frac{1}{\sqrt{(i-1)!}} \frac{(-x)\phi(x)}{2\sqrt{\phi(x)}} f_{i-1}(x) + \sqrt{\frac{\phi(x)}{(i-1)!}} \big(x f_{i-1}(x) - f_i(x)\big) \\
&=  \frac{1}{2} x e_i(x)  - \sqrt{i} \, e_{i+1}(x) =  \frac{1}{2} \big( \sqrt{i-1} \, e_{i-1}(x) - \sqrt{i} \, e_{i+1}(x))
\intertext{and}
e_i ^{\prime \prime}  &= \frac{1}{2}\big( \sqrt{i-1} \, e_{i-1}^\prime - \sqrt{i} \, e_{i+1}^\prime)
= \frac{1}{4} \big( \sqrt{(i-1)(i-2)} \, e_{i-2} +(1-2i) \, e_i + \sqrt{i(i+1)}\, e_{i+2}).
\end{align*}
We conclude by noting that
\begin{align*}
x\, (e_i(x)) ^\prime &=   \frac{1}{2}\big( \sqrt{i-1} \, x e_{i-1} - \sqrt{i} \, x e_{i+1} \big) \\
&=  \frac{1}{2}\big( \sqrt{(i-1)(i-2)}\, e_{i-2} -e_i - \sqrt{i(i+1)}\, e_{i+2}\big).
\end{align*}

\end{proof}

\begin{proof}[Proof of Lemma~\ref{lemma:coefficients-kalman}]
We start by observing that
\begin{align*}
\ccL e_j(x) &= bx\, e_j^\prime(x) + \frac{\sigma^2}{2}\, e_j^{\prime\prime}(x) \\
&= \frac{b}{2} \big( \sqrt{(j-1)(j-2)}\, e_{j-2} -e_j - \sqrt{j(j+1)}\, e_{j+2} \big) \\
&+ \frac{\sigma^2}{8} \big( \sqrt{(j-1)(j-2)} \, e_{j-2} + (1-2j) \, e_j + \sqrt{j(j+1)} \, e_{j+2}  \big).
\end{align*}
The expression for $(e_i,\ccL e_j)$ now follows by orthonormality of the Hermite basis. In a similar way
\begin{align*}
h \, x \, e_j(x) &= h \big( \sqrt{j-1} \, e_{j-1}(x) + \sqrt{j} \, e_{j+1} (x) \big),
\end{align*}
and the second expression follows. Finally, note that
\begin{align*}
(\lambda x^2 -1 ) e_j(x) &= \lambda \big(\sqrt{(j-1)(j-2)} \, e_{j-2}(x) +(2j-1) e_j(x) + \sqrt{j(j+1)} \, e_{j+2}(x) \big) - e_j(x)
\end{align*}
and we conclude.
\end{proof}

\begin{proof}[Proof of Lemma~\ref{lemma:hermite-moments}] First, note that,
\begin{align*}
({1},e_{i+1})
	&= \frac{1}{\sqrt{i!}}\int (2\pi)^{-\frac{1}{4}}e^{-\frac{x^2}{4}} f_{i}(x)dx\\
	&= \frac{\sqrt{2} (2\pi)^{-\frac{1}{4}}}{\sqrt{i!}}  \int e^{-\frac{x^2}{2}} f_i(\sqrt{2}x)dx\\
	&= \frac{\sqrt{2} (2\pi)^{-\frac{1}{4}}}{\sqrt{i!}} \int e^{-\frac{x^2}{2}} \sum_{k=0}^i \vartheta^i_k  {2}^\frac{k}{2} x^k dx\\
	&= \frac{\sqrt{2} (2\pi)^{\frac{1}{4}}}{\sqrt{i!}} \int \frac{1}{\sqrt{2\pi}}e^{-\frac{x^2}{2}} \sum_{k=0}^i \vartheta^i_k  {2}^\frac{k}{2} \sum_{j=0}^k \iota^k_j f_j(x) dx\\
	&= \frac{\sqrt{2} (2\pi)^{\frac{1}{4}}}{\sqrt{i!}}  \sum_{k=0}^i \sum_{j=0}^k \vartheta^i_k  {2}^\frac{k}{2} \iota^k_j \int \frac{1}{\sqrt{2\pi}}e^{-\frac{x^2}{2}} f_j(x)f_0(x) dx\\
	&= \frac{\sqrt{2} (2\pi)^{\frac{1}{4}}}{\sqrt{i!}}  \sum_{k=0}^i \vartheta^i_k {2}^\frac{k}{2}  \iota^k_0 .
\end{align*}
 since 
any power of $x$ can be represented as linear combination of the Hermite polynomials. The last step follows by orthonormality of the Hermite basis. An analogous argument
with the constant function $1$ replaced by $x^j$ gives the result.
\end{proof}

\begin{proof}[Proof of Lemma~\ref{lemma:projection-of-initial-density}]
The main difficulty is the noncentrality of $q_0$. Our main tool are the representations
$f_j(x)=\sum_{k=0}^j \vartheta^j_k x^k$
 and  $x^i = \sum_{k=0}^i \iota^i_k f_k(x).$
 For arbitrary $a$ and $b$, a non-central Hermite polynomial has a representation in terms of central Hermite polynomials as follows:
	\begin{align*}
		f_j(a + b x) &= \sum_{k=0}^j \vartheta^j_k \, (a + b x)^k \\
		&= \sum_{k=0}^j \vartheta^j_k \, \sum_{m=0}^k {k \choose m} \, a^{k-m} \, b ^m \, x^m.
	\end{align*}
	Using the representation $x^m = \sum_{r=0}^m \iota^m_r f_r(x) $ leads to
	\begin{align*}
		f_j(a + b x) &= \sum_{k=0}^j \sum_{m=0}^k \sum_{r=0}^m    \vartheta^j_k \, {k \choose m} \, a^{k-m} \, b ^m \iota^m_r f_r(x) \\
		&= \sum_{r=0}^j\Big( \sum_{m=r}^j \sum_{k=m}^j  \vartheta^j_k \, {k \choose m} \, a^{k-m} \, b ^m \iota^m_r \Big) f_r(x) \\
		&=: \sum_{r=0}^j A_r(j,a,b) f_r(x).
	\end{align*}
	To obtain our result, we start from
	\begin{align*}
		( q_0, e_j ) &= \int \frac{1}{\sqrt{2 \pi \sigma_0^2 }}e^{-\frac{(x-\mu_0)^2}{2 \sigma_0^2}} \frac{1}{\sqrt{(j-1)!}} \frac{1}{(2\pi)^{-1/4}} e^{-\frac{x^2}{4}} f_{j-1}(x) dx.
	\end{align*}
	Observe that
	$$ \frac{(x-\mu_0)^2}{2 \sigma_0^2}+ \frac{x^2}{4} = \frac{(x-a)^2}{2 b^2} - d $$
	with $a$, $b$ and $d$ as specified in the Lemma. Furthermore,
	\begin{align*}
		\frac{1}{\sqrt{2\pi }} \int e^{-\frac{(x-a)^2}{2 b^2}}  f_{j-1}(x) dx & = \int \frac{1}{\sqrt{2\pi}} e^{-\frac{x^2}{2}} f_{j-1}(bx+a) dx  \\
		&= \sum_{r=0}^{j-1} A_r (j-1,a,b) \int \phi(x) \,  f_r(x) f_0(x) dx \\
		&= A_0 (j-1,a,b)
	\end{align*}
	by orthogonality of the Hermite polynomials.  Summarizing, we obtain that
	\begin{align*}
		( q_0, e_j ) &= \frac{1}{\sqrt{ \sigma_0^2 (j-1)!}}\frac{e^d}{(2\pi)^{-1/4}}   A_0(j-1,a,b) \\
		 &=  \frac{1}{\sqrt{ \sigma_0^2(j-1)!}} \frac{e^d}{(2\pi)^{-1/4}} \sum_{m=0}^{j-1} \sum_{k=m}^{j-1}  \vartheta^{j-1}_k \, {k \choose m} \, a^{k-m} \, b^m \iota_0^m.
	\end{align*}

\end{proof}

\bibliographystyle{chicago}

\end{document}